\documentclass[twoside]{article}
\usepackage{graphicx}
\usepackage{amsmath}
\usepackage{amssymb}
\usepackage{amscd}
\usepackage{hhline}
\usepackage{epic}
\usepackage{mathrsfs}
\usepackage{dsfont}

\newtheorem{theorem}{Theorem}

\newtheorem{theoremc}{Theorem}
\newtheorem{theoremd}{Theorem}
\newtheorem{theoreme}{Theorem}

\newtheorem{rk}[theoremc]{Remark}
\newtheorem{cor}[theoremd]{Corollary}
\newtheorem{lem}[theoreme]{Lemma}

\newtheorem{prop}[theorem]{Proposition}
\newenvironment{proof}[1][Proof]{\textbf{#1.} }{\qed \vspace{5pt}}
\newcommand\bib[1]{\bibitem[#1]{#1}}

\newcommand\abz{\hspace{13.5pt}}
\newcommand\qed{\phantom{\underline{y}}\hfill\hfill$\square$}

\renewcommand\a{\alpha}
\renewcommand\b{\beta}

\newcommand\C{{\mathbb C}}

\renewcommand\d{\delta}
\newcommand\D{{\cal D}}

\newcommand\E{{\mathcal E}}

\newcommand\g{\gamma}

\newcommand\hps{\hskip-16pt . \hskip2pt}

\renewcommand\l{\lambda}

\newcommand\La{\Lambda}

\newcommand\op[1]{\mathop{\rm #1}\nolimits}
\newcommand\ot{\otimes}
\newcommand\p{\partial}

\newcommand\R{{\mathbb R}}
\renewcommand\t{\tau}

\newcommand\ti{\tilde}

\newcommand\vp{\varphi}
\newcommand\ve{\varepsilon}

\newcommand\z{\sigma}

\newcommand\imag{\mathrm{i}}

\DeclareFontFamily{U}{wncyr}{}
\DeclareFontShape{U}{wncyr}{m}{it}{%
  <5><6><7><8><9>gen*wncyi%
  <10><10.95><12><14.4><17.28><20.74><24.88>wncyi10}{}
\DeclareSymbolFont{MathRussLetters}{U}{wncyr}{m}{it}
\DeclareMathSymbol{\re}{\mathalpha}{MathRussLetters}{3}

\makeatletter
\renewcommand{\@oddhead}{\hfil Liouville metrics\hfil}
\renewcommand{\@evenhead}{\hfil Boris Kruglikov\hfil}
\renewcommand{\@begintheorem}[2]{\begin{trivlist}\it
 \item[\hspace{\labelsep}{\bf #1\ #2.}]}
\renewcommand{\@endtheorem}{\end{trivlist}}
\makeatother

\newcommand\Cc{\let\mathcal\mathscr\cal C}
\newcommand\Jj{\let\mathcal\mathscr\mathcal J}
\newcommand\Pp{\let\mathcal\mathscr\mathcal P}

\begin{document}

 % \title{An effective compatibility criterion\\for systems of PDEs via multi-brackets}
 \title{Invariant characterization of Liouville metrics\\
 and polynomial integrals}
 \author{Boris Kruglikov}
 \date{}
 \maketitle

 \vspace{-14.5pt}
 \begin{abstract}
A criterion in terms of differential invariants for a metric on a
surface to be Liouville is established. Moreover, in this paper we
completely solve in invariant terms the local mobility problem of a
2D metric, considered by Darboux: How many quadratic in momenta
integrals does the geodesic flow of a given metric possess? The
method is also applied to recognition of other polynomial integrals
of geodesic flows.
 \end{abstract}

%%%%%%%%%%%%%%%%%%%%%%%%%%%%%%%%%%%%%%%%%%%%%%%%%%%%%%%%%%%%%%%%%%%%%%%%%%%%
%0%
\section*{Introduction}\label{S0}

 \hspace{13.5pt}
The problem of recognizing by a metric, how many integrals admits
its geodesic flow is classical. In this paper we study
locally metrics on surfaces. We will look for the integrals
analytic in momenta.

By Whittaker theorem \cite{W} existence of such
an integral is equivalent to existence of an integral polynomial in
momenta. Note that locally geodesic flows are integrable, but the
corresponding integrals are usually analytic only on $T^*M\setminus M$.
So in general polynomial integrability requires certain conditions even
locally.

The integrals of degree one in momenta correspond to surfaces of
revolution, locally $ds^2=f(x)(dx^2+dy^2)$. It is an easy fact that if such integrals exist,
then there are
either one (generically) or three (space form). We provide a
precise criterion for determining existence of a local linear integral
(Killing vector field).

The next interesting case concerns geodesic flows with quadratic in
momenta integrals. They correspond to Liouville metrics. The local
analytic form of such metrics near a generic point is well-known
\cite{D,B}: $ds^2=(f(x)+h(y))(dx^2+dy^2)$ and a metric has an
additional quadratic integral iff it can be transformed into such a
form.

However no criterion, when the metric
is Liouville has been previously obtained, except for the paper \cite{Su}
(the first, rather unsuccessful attempt was done in \cite{V}).
However the criterion of this work was not explicit. Neither did it contain invariant
formulae, making it difficult even to decide how many differential invariants
characterize Liouville metric and which order they have.

As the main result of this paper we resolve the classical problem of
recognition for Liouville metrics and provide an explicit criterion
written via a basis of scalar differential invariants of the metric.

Moreover we shall determine the number of quadratic integrals, which
coincides with the degree of mobility of a Riemannian metric on a
surface. It can be 6 (space form), 4 (case characterized in \cite{D}
though not through differential invariants; note that in classical
works the Hamiltonian is disregarded, so that there are 5, 3 etc
integrals in their way of counting), 3 (the case studied by K\oe
ning \cite{Koe}), 2 (general Liouville form) or 1 (metrics with no
additional quadratic integrals). Each of these cases will be
characterized via an invariant condition written in terms of
differential invariants of the Riemannian metric.

The method of our study is the Cartan's prolongation-projection
method: we write the system of PDEs for existence of a quadratic
integral and subsequently calculate the compatibility conditions. If
they are trivial, the system is compatible and we stop. Otherwise we
add new equations, the space of solutions (which is a
finite-dimensional linear space from the beginning -- the system is
of finite type) shrinks and we continue.

For effectiveness of the method we should have explicit formulas
for compatibility conditions, but they are given by the result of
\cite{KL$_2$}.

The procedure stops in several steps because finally we arrive to
only one possible quadratic integral, which is just the
Hamiltonian, an obvious integral of the geodesic flow. The
prolongation-projection scheme usually is characterized by the
rapidly growing complexity with each step. It is also true in our
problem, but in this case we manage to arrive to the very end of
the method and to establish the solvability criterion.

The problem of invariant characterization of Liouville metrics was
initiated in paper \cite{KL$_3$} in a collaboration with V.Lychagin
as an application of our general compatibility criterion.
The results are repeated in a revised form in sections \ref{S3}-\ref{S4}.
Moreover the general idea of solution to the problem was sketched there, but the
complete answer appears here for the first time.

Let us also indicate that the solution of the problem presented here
is expressed via scalar differential invariants and we especially
care to minimize the number/order of the invariants.

At the end of the paper we discuss the problem of higher degree integrals
and make some claims and conjectures about dimension of the space $\Jj_n$
of polynomial in momenta integrals $F$ of $\op{deg}F=n$.

\smallskip

{\bf Acknowledgment.} While solving the problem I profited from
discussions with V. Matveev, V. Lychagin and E. Ferapontov. I wish to thank them all.

%%%%%%%%%%%%%%%%%%%%%%%%%%%%%%%%%%%%%%%%%%%%%%%%%%%%%%%%%%%%%%%%%%%%%%%%%%
\section{\hps PDEs and Prolongation-Projection scheme}\label{S1}

 \abz
In this section we deduce the basic system of equations for polynomial
integrability, discuss compatibility criterions and formulate the general
scheme of investigating solvability.

This paper deals mostly with local existence problem, so that
whenever opposite is not explicitly stated all statements should be
assumed local in $M$. Moreover we impose the usual regularity
assumption. Regular points form a nonempty open set, but it does not
need to be of full measure in the $C^\infty$ case (in analytic case
regular points are generic) and such pathological examples exist
even in $\R^2$.

Let $(x,y)$ be local coordinates on $M^2$ and $p_x,p_y$ be the
corresponding momenta on $T^*M$. Writing the metric
$ds^2=g_{ij}dx^idx^j$ we express Hamiltonian of the geodesic flow as
$H=g^{11}p_x^2+2g^{12}p_xp_y+g^{22}p_y^2$, where the matrix $g^{ij}$
is inverse to the matrix $g_{ij}$ of the metric $g$.

A homogeneous term of an integral is obviously an integral, and so
we study a function $F_n=\sum_{i+j=n}u_{ij}(x,y)p_x^ip_y^j$ on
$T^*M$. Involutivity condition, i.e. vanishing of the Poisson
bracket $\{H,F_n\}=0$, is equivalent to $(n+2)$ equations
$E_1=0,\dots,E_{n+2}=0$ on $(n+1)$ unknown function
$u_{n0}(x,y),\dots,u_{0n}(x,y)$.

This system $\E$ is of generalized complete intersection type studied in \cite{KL$_3$}.
The compatibility criterion developed there states that the system is formally integrable iff the following multi-bracket vanishes:
 \begin{equation}\label{multbr}
E_{n+3}=\{E_1,\dots,E_{n+2}\}\ (\op{mod}E_1,\dots,E_{n+2})=0.
 \end{equation}
For linear differential operators this is defined as follows.

Let $E_i(u)=\sum_{j=0}^n E_i^ju_{j(n-j)}$ be representation of the vector operator in components $E_i=(E_i^0,\dots,E_i^n)$ ($E_i^j$ are scalar differential operators). Then the multi-bracket equals
(in this formula $m=n+1$)
 $$
\{E_{1},\dots,E_{m+1}\}=
\!\dfrac{1}{m!}\hspace{-0.08in}
\sum_{\alpha\in\mathbf{S}_{m},\beta\in\mathbf{S}_{m+1}}\hspace{-0.08in}(-1)^\alpha(-1)^\beta\, E^{\alpha(0)}_{\beta(1)}\!\cdot\!E^{\alpha(1)}_{\beta(2)}\cdots E^{\alpha(m-1)}_{\beta(m)}\!\cdot E_{\beta(m+1)}.
 $$
Reduction modulo the system in (\ref{multbr}) means the following. Orders of differential operators $E_i$ are 1 and order of the multi-bracket $\{E_1,\dots,E_{n+2}\}$ is (no greater than) $(n+1)$. We prolong the system $\E=\{E_1=0,\dots,E_{n+2}=0\}$ to the order $(n+1)$ (such prolongation exists!), i.e. take the space of all linear combinations $\sum_{i=1}^{n+2}\nabla_i E_i$ with linear scalar differential operators $\nabla_i$ of $\op{ord}\nabla_i\le n$, and consider the class of the multi-bracket in the quotient space (see \cite{KL$_2$} for details).

Now the system $\E$ is of finite type, i.e. has no complex
characteristics (we refer the reader for this and further notions
from geometry of PDEs to \cite{KLV,KL$_4$}), so if compatibility
condition (\ref{multbr}) is satisfied, the system is locally
integrable.

If this condition is not satisfied we add the equation $E_{n+3}=0$
to the system $\E$ and continue with solvability investigation of
the new prolonged system $\E'$. Prolongation means addition of
derivatives of the generators. But if their combination
(differential corollary) drops in order (projection), it is the
compatibility condition, which should be added to the system. Thus
we get new equations $\E''$ etc, until we stabilize at a system
$\bar\E$ in a finite number of steps (Cartan-Kuranishi theorem).

Denoting $T=T_{(x,y)}M$ the model tangent space for independent variables and  $N\simeq\R^{n+2}(u_{i(n-i)})$ the space of dependent variables, the system $\E$ has symbols $g_k\subset S^kT^*\ot N$ (respectively $\E'$ has symbols $g'_k$ etc). Codimension of this linear subspace $(n+2)(k+1)-\dim g_k$ equals the number of independent equations in the prolonged system of order $k$.

Cohomology $H^{k-1,2}(\E)$ of the Spencer $\d$-complex
 $$
0\to g_{k+1}\stackrel{\d}\longrightarrow g_k\ot T^*\stackrel{\d}\longrightarrow g_{k-1}\ot\La^2T^*\to0
 $$
at the last term is the space of compatibility conditions (\cite{KLV,KL$_4$}). For the initial system $\E$ the only non-zero second Spencer $\d$-cohomology is $H^{n,2}(\E)\simeq\R^1$ and the only compatibility condition is the above reduced multi-bracket $E_{n+3}$. Thus $\E'=\{E_1=0,\dots,E_{n+3}=0\}$. Further compatibility conditions (for $\E'$ etc) will be indicated along with prolongation-projection process.

Let $\Jj_n=\{(u_{i(n-i)})\}$  denotes the solution space of the system $\E$. This space is linear and $\dim\Jj_n=\sum\dim\bar g_k$. In particular this is smaller than $\sum\dim g_k$
and since the latter quantity strictly decreases during prolongation-projection, the
method stops in a finite number of steps giving either non-trivial locally integrable system of PDEs (solvability) or no solution result.

%%%%%%%%%%%%%%%%%%%%%%%%%%%%%%%%%%%%%%%%%%%%%%%%%%%%%%%%%%%%%%%%%%%%%%%%%%
\section{\hps Differential invariants of a Riemannian metric}\label{S2}

 \abz
Here we describe the algebra $\mathfrak{A}$ of scalar differential invariants of a
Riemannian metric $g$ on a two-dimensional surface $M$. It is well-known that the first such invariant occurs in order 2 and is given by the scalar curvature $K$.

Denote $\op{grad}K$ be the $g$-gradient of the curvature and let
$\op{sgrad}K=J_0\op{grad}K$ be its rotation by $\pi/2$ (one needs to fix orientation, which is possible as we treat $(M,g)$ locally; alternatively we can square those invariants, which have undetermined sign).
There are two invariant differentiations $\mathfrak{L}_{\op{grad}K}$ and $\mathfrak{L}_{\op{sgrad}K}$ ($\mathfrak{L}$ is the Lie derivative). These differentiations and $K$ do not generate $\mathfrak{A}$, but if we also allow commutators they do.

There are precisely $(k-1-\delta_{k3})$ functionally independent differential invariants of order $k$ for $k>0$. Let us briefly explain why (cf. \cite{T}). Consider the jet-space of Riemannian metrics $J^k(S^2_+T^*M)$. Fibers of the projections $\pi_{k,k-1}:J^k\to J^{k-1}$ have dimensions $3(k+1)$, where as usual $J^{-1}=M$.

The pseudogroup $\op{Diff}_\text{loc}(M)$ acts naturally on the jet-spaces. Its action is transitive up to 1st jets. Indeed, the action is clearly transitive on the base and let us consider the jets of the stabilizer of $x\in M$, i.e. the differential group $G_x^k=J^k_{x,x}(M,M)$. $G_x^1$ acts transitively on $J^0_x(S^2_+T^*M)$ with one-dimensional stabilizer $O(2)$.

The action of $G^2_x$ on $J^1_x(S^2_+T^*M)$ is transitive as well, but the action of $G^3_x$ on $J^2_x(S^2_+T^*M)$ is not (though it has dimensional freedom to be!). Codimension of a generic orbit is 1 and curvature $K$ is the only invariant.

The stabilizer disappears only on the next step and starting from $k=3$ the actions of $G^{k+1}_x$ on $J^k_x(S^2_+T^*M)$ are free. Thus since $\dim\op{Ker}(G^{k+1}_x\to G^k_x)=2(k+2)$, codimension of the generic orbit of $\op{Ker}(G^{k+1}_x\to G^k_x)$ on $\pi_{k,k-1}^{-1}(*)$ is equal to $(k-1-\delta_{k3})$.

Let us establish a basis in the space of invariants.

For a tensor $T$ denote $d_\nabla^{\ot k}T=d_\nabla(d_\nabla^{\ot(k-1)}T)$
the iterated covariant derivative of the tensor $T$
($d_\nabla^{\ot2}$ differs from $d_\nabla^2$, which is equal to
multiplication by the curvature tensor). In particular, we obtain
the forms $d_\nabla^{\ot i}K\in C^\infty(\ot^iT^*M)$.

Since $K$ is a scalar function, $d_\nabla K=dK$. The next differential is symmetric, because we consider metric (Levi-Civita) connection:

 \begin{lem}
$d^{\ot2}_\nabla F\in C^\infty(S^2T^*M)$ $\forall F\in C^\infty(M)$ iff connection $\nabla$ is symmetric.
 \end{lem}

 \begin{proof}
Since $d^{\ot2}_\nabla F(\xi,\eta)=(\nabla_\xi dF)(\eta)=\nabla_\xi[\eta(F)]-[\nabla_\xi\eta](F)$, we have:
 $$
d^{\ot2}_\nabla F(\xi,\eta)-d^{\ot2}_\nabla F(\eta,\xi)=
\bigl(L_{[\xi,\eta]}-(\nabla_\xi\eta-\nabla_\eta\xi)\bigr)(F)=T_\nabla(\eta,\xi)(F),
 $$
where $T_\nabla$ is the torsion tensor. In coordinates this is expressed via Christoffel symbols as $(d^{\ot2}_\nabla F)_{ij}=F_{ij}-\Gamma^l_{ij}F_l$, where $F_\z=\tfrac{\p^{|\z|}F}{\p x^\z}$ are the partial derivatives.
 \end{proof}

The next differential $d^{\ot3}_\nabla F\in\Omega^1M\ot S^2\Omega^1M$, but this (and higher) tensors are fully symmetric iff the metric is flat:

 \begin{lem}
Let $T_\nabla=0$. Then $d^{\ot3}_\nabla F\in C^\infty(S^3T^*M)$ $\forall F\in C^\infty(M)$ iff $R_\nabla=0$.
 \end{lem}

 \begin{proof}
$d^{\ot3}_\nabla F(\xi,\eta,\theta)=(\nabla_\xi\nabla_\eta dF-\nabla_{\nabla_\xi\eta} dF)(\theta)$, whence
 $$
d^{\ot3}_\nabla F(\xi,\eta,\theta)-d^{\ot3}_\nabla F(\eta,\xi,\theta)= \bigl(([\nabla_\xi,\nabla_\eta]-\nabla_{[\xi,\eta]})dF\bigr)(\theta)=
R_\nabla(\xi,\eta)^*dF(\theta)
 $$
and the result follows.
 \end{proof}

Now to fix a basis in invariants of order $i=2+l$ we consider the form
$d_\nabla^{\ot l}K$ and denote (in non-flat case the order is essential!)
 $$
I_{ij}=d_\nabla^{\ot l}K(\underbrace{\op{grad}K,\dots,\op{grad}K}_{l-j},
\underbrace{\op{sgrad}K,\dots,\op{sgrad}K}_{j}).
 $$
If we change the order, the expression will be changed by a lower order differential invariant. We will not use it and so omit the details, but for instance
 $$
d_\nabla^{\ot3}K(\op{sgrad}K,\op{sgrad}K,\op{grad}K)-
d_\nabla^{\ot3}K(\op{grad}K,\op{sgrad}K,\op{sgrad}K)=|\op{grad}K|^2.
 $$

The first invariants are: $I_2=K$ and $I_3=|\nabla K|^2$ (the
index refers to the order of differential invariant). Starting from $i=4$ there are $l+1=i-1$ invariants $I_{ij}$ and we re-enumerate the index $j$ by letters (so we write $I_{4a}$ instead of $I_{40}$, $I_{5d}$ instead of $I_{53}$ etc). For instance $I_{4b}=d_\nabla^{\ot2}K(\op{grad}K,\op{sgrad}K)$.

The two approaches to describe the algebra $\mathfrak{A}$ of differential invariants, one via the basic invariant $I_2$ with two invariant differentiations and another one via the basis $I_{ij}$ are closely related: the former is obtained from the latter via the Lie-Tresse approach \cite{Tr}. Namely let say $I_2,I_3$ be chosen as a basis, and $\hat\p/\hat\p_{I_2},\hat\p/\hat\p_{I_3}$ be the corresponding Tresse derivatives (see \cite{KL$_1$}). In local coordinates $(x^1,x^2)$ they can be expressed as
 $$
\hat\p/\hat\p_{I_2}=\Delta^{-1}\bigl(\D_2(I_3)\D_1-\D_1(I_3)\D_2\bigr),\quad \hat\p/\hat\p_{I_3}=\Delta^{-1}\bigl(\D_1(I_2)\D_2-\D_2(I_2)\D_1\bigr),
 $$
where $\D_i$ are total derivatives \cite{KLV} and $\Delta=\D_1(I_2)\D_2(I_3)-\D_2(I_2)\D_1(I_3)$ is the determinant (basis requirement above means $\Delta\not\equiv0$). Then the two invariant differentiations $\nabla_1=\mathfrak{L}_{\op{grad}K}$ and $\nabla_2=\mathfrak{L}_{\op{sgrad}K}$ equal
 $$
\nabla_1=I_3\cdot\hat\p/\hat\p_{I_2}+2I_{4a}\cdot\hat\p/\hat\p_{I_3},\quad \nabla_2=2I_{4b}\cdot\hat\p/\hat\p_{I_3}.
 $$

Relation to the other side constitutes an infinite sequence of identities:
 \begin{gather*}
\nabla_1\,I_2=I_3,\ \nabla_2\,I_2=0,\quad
\nabla_1\,I_3=2I_{4a},\ \nabla_2\,I_3=2I_{4b},\\
\nabla_1\,I_{4a}=I_{5a}+\frac{2(I_{4a}^2+I_{4b}^2)}{I_3},\
\nabla_2\,I_{4a}=I_{5b}+\frac{2I_{4b}(I_{4a}+I_{4c})}{I_3},\\
\nabla_1\,I_{4b}=I_{5b}+\frac{I_{4b}(I_{4a}+I_{4c})}{I_3},\
\nabla_2\,I_{4b}=I_{5c}+\frac{I_{4c}^2-I_{4a}I_{4c}+2I_{4b}^2}{I_3}+I_2I_3^2,\\
\nabla_1\,I_{4c}=I_{5c}+\frac{2(I_{4a}I_{4c}-I_{4b}^2)}{I_3},\
\nabla_2\,I_{4c}=I_{5d},\qquad\dots
 \end{gather*}
They can be obtained successively with the help of the commutation rule for invariant differentiations:
 $$
[\op{grad}K,\op{sgrad}K]=
-\frac{2I_{4b}}{I_3}\op{grad}K+\frac{I_{4a}-I_{4c}}{I_3}\op{sgrad}K.
 $$

%%%%%%%%%%%%
\section{\hps Linear integrals}\label{S3}

 \abz
A Riemannian metric $g$ on a surface $M^2$ possesses a Killing
vector field iff it has the following local form near the point,
where the field does not vanish:
$ds^2=g_{11}(x)dx^2+2g_{12}(x)dxdy+g_{22}(x)dy^2$, so that $(M^2,g)$ is a
surface of revolution. How to recognize such a metric?

Let us write the metric locally in isothermal hyperbolic coordinates (possibly over $\C$):
$ds^2=e^{\l(x,y)}dxdy$. If the metric is positive definite (not pseudo-Riemannian), one should rather write $e^{\l}dzd\bar{z}$ and this
complexification pop-ups as follows: while the gradient of a function $K$
equals $(2e^{-\l}K_y,2e^{-\l}K_x)$, the skew-gradient is
$(2\mathrm{i}\,e^{-\l}K_y,-2\mathrm{i}\,e^{-\l}K_x)$! Moreover we shall encounter $\mathrm{i}$ as a factor at some coefficients below, but this does not lead to contradiction: vanishing of these coefficients turns out to be a real condition.

In \cite{KL$_3$} we chose the general form, but since the answer will be expressed in differential invariants, the choice is not essential.

Function $F_1=u p_x+v p_y$ is an integral of the geodesic flow iff
the following 3 linear PDEs (coefficients of $\{H,F_1\}$) are satisfied:
 $$
u_y=0,\qquad u_x+v_y+u\l_x+v\l_y=0,\qquad v_x=0.
 $$
Denote them by $E_1,E_2,E_3$ respectively. This system $\E$ has symbols:
$\dim g_0=2$, $\dim g_1=1$, $\dim g_2=0$. The compatibility
condition is given by the relation
 $$
E_4=\{E_1,E_2,E_3\}\ (\op{mod} E_1,E_2,E_3)=0.
 $$
In general case the bracket should have order 2 in pure form and 1 after reduction,
but in our case $E_4$ is of order 0 and equals:
 $$
E_4=\tfrac12e^\l(K_xu+K_yv),
 $$
where $K$ is the Gaussian curvature. Thus compatibility condition means
$(M^2,g)$ is a spatial form: $K=\op{const}$. This is the case, when $\dim\Jj_1=3$.

If $K$ is non-constant, to study solvability we add the equation
$E_4=0$ to the system. To describe the new system $\E'$ we let
$u=K_y w$, $v=-K_x w$ and obtain the following system on one function $w(x,y)$:
 $$
\begin{pmatrix}
0 & K_y & K_{yy} \\ -K_x & 0 & -K_{xx} \\ K_y & -K_x & \l_xK_y-\l_yK_x
\end{pmatrix}\cdot
\begin{bmatrix}
w_x \\ w_y \\ w
\end{bmatrix}=0.
 $$

In order to have solutions the determinant of this matrix should vanish. It equals $-\frac{\mathrm{i}}4e^{2\l}I_{4b}$. Given this condition we can drop one equation and transform the system to the form
 $$
(\log|K_x\,w|)_x=0,\qquad (\log|K_y\,w|)_y=0.
 $$
Its solvability is equivalent to a 3rd order relation on the curvature, which can be expressed as the condition $I_3(I_{5b}+I_{5d})=2I_{4b}(I_{4a}+I_{4c})$. However when $I_{4b}=0$, then $I_{5b}=0$ and we obtain:

 \begin{theorem}\label{KVF}
$\dim\Jj_1=3$ iff $K=\op{const}$ (i.e. $I_3=0$) and $\dim\Jj_1=1$ iff
 $$
I_{4b}=0,\ I_{5d}=0.
 $$
Otherwise there exist no local Killing vector fields.
 \end{theorem}

 \begin{rk}
This and further statements hold only near regular points (here this
means $dK\ne0$). Indeed in non-analytic case there exist
pathological counterexamples. For instance for any $\ve>0$ it is
possible to construct a $C^\infty$-metric on the disk $D^2(1)$
satisfying $I_{4b}=I_{5d}=0$, such that the set of regular points
(where a Killing field exist) has Lebesgue measure $<\ve$.
 \end{rk}

We can reformulate this criterion as vanishing of the differential invariants $\op{Jac}_g(K,|\nabla K|^2)$ and $\op{Jac}_g(K,\Delta_gK)$, where $\op{Jac}_g(F,G)=dF\wedge dG\bigl(\frac{\op{grad}K}{|\nabla K|},\frac{\op{sgrad}K}{|\nabla K|}\bigr)$ is the Jacobian and $\Delta_g F=\op{Tr}_g[d_\nabla^{\ot2}F]$ is the Laplacian. Indeed we have: $\Delta_gK=(I_{4a}+I_{4c})/I_3$, so the claim follows from:
 $$
\op{Jac}_g(K,|\nabla K|^2)=2I_{4b},\qquad \op{Jac}_g(K,\Delta_gK)=\frac{I_{5b}+I_{5d}}{I_3}
 $$
(note that $I_{4b}=0$ implies $I_{5b}=0$).

 \begin{rk}
Some classical criteria for existence of local (global implications
follow) Killing fields are contained in \cite{Nij,Nom}, but they are
neither explicit conditions on the metric $g$ nor finitely
formulated. Our criterion in the form of dependence of $|\nabla K|$
and $\Delta_gK$ on $K$ is implicitly contained in \cite{D}.
 \end{rk}

% \begin{rk}
% A criterion of this kind seems to should have been known, but we did not find it in the literature. Some
% classical criterions for existence of local (global implications follow) Killing fields are contained in
% \cite{Nij,Nom}, but they are neither explicit conditions of the metric $g$ nor finitely formulated.
% \end{rk}

%%%%%%%%%%%%
\section{\hps More than 3 quadratic integrals}\label{S4}

 \abz
We turn now to characterization of Liouville metrics. We will again use isometric hyperbolic coordinates, $H=e^{-\l}p_xp_y$, which does not restrict generality.

The function $F_2=u(x,y)p_x^2+2v(x,y)p_xp_y+w(x,y)p_y^2$ is a quadratic
integral of the geodesic flow iff the following system $\E$ is satisfied:
 $$
u_y=0,\ \ u_x+2v_y+2u\l_x+2v\l_y=0,\ \ 2v_x+w_y+2v\l_x+2w\l_y=0,\ \ w_x=0.
 $$
Denote the equations respectively by $E_1,E_2,E_3,E_4$. The
compatibility condition can be expressed via the multi-bracket
 $$
E_5'=\{E_1,E_2,E_3,E_4\}\ (\op{mod} E_1,E_2,E_3,E_4)=0.
 $$
Even though it might be expected from the general theory that
$E_5'$ has order 2, in our case it has order 1. Divided by $2e^{\l}$ it
equals to
 $$
E_5=5K_x v_y-5K_y v_x-(K_{xx}-\l_xK_x)u+5(\l_yK_x-\l_xK_y)v+(K_{yy}-\l_yK_y)w.
 $$

Thus the system $\E$ is formally integrable iff $K=\op{const}$. In
this case $\dim g_0=3,\dim g_1=2,\dim g_2=1$, $g_{2+i}=0$ for $i>0$ and
the dimension of the solutions space is $\dim\Jj_2=\sum\dim g_k=6$.
Indeed $\Jj_2=S^2\Jj_1$, i.e. a basis in the space of quadratic integrals is formed by
pair-wise products of elements of a basis in is the space of linear integrals.

Suppose that $K\ne\op{const}$, so that $E_5$ is a differential relation of the
first order in $u,v,w$. Adding $E_5=0$ we get the system\footnote{$J^k(m,r)$ is the space of $k$-jets of maps
$\vp:\R^m\to\R^r$ and formal codimension of a system $\E\subset J^k(m,r)$ is $\sum_i\op{dim}H^{1,i}(\E)$, the precise number of the equations in the system \cite{KL$_4$}.}
$\E'\subset J^1(2,3)$ of formal codimension 5.

Its symbols $g'_i\subset S^iT^*\ot\R^3$ have $\dim g_0'=3$, $\dim
g_1'=1$, $\dim g_2'=0$ and thus the only non-zero second
$\d$-cohomology groups\footnote{The second Spencer $\d$-cohomology $H^{*,2}=\oplus H^{i-1,2}$ is the
space that contains all compatibility conditions of the system. The latter are called Weyl tensors $W_i\in H^{i-1,2}$ (also called curvatures/torsions/structural functions). We refer to \cite{KL$_4$} for a review.}
are $H^{0,2}(\E')\simeq\R^1$, $H^{1,2}(\E')\simeq
\R^1$. There are two obstructions to compatibility -- Weyl tensors
$W_1'\in H^{0,2}(\E')$ and $W_2'\in H^{1,2}(\E')$. The former $W_1'$ is proportional to
  $$
\! E_6'=K_yE_{5x}+K_xE_{5y}-\frac52K_x^2(E_{2y}-E_{1x})
+\frac52K_y^2(E_{3x}-E_{4y})(\op{mod} E_1,E_2,E_3,E_4,E_5).
 $$
Multiplying this by $5K_x$ and further simplifying modulo
$E_1,E_2,E_3,E_4,E_5$ we obtain the following expression:
 $$
E_6= \tfrac{35}{4\mathrm{i}}e^{2\l}\,I_{4b}\,v_x+Q_1\,u+
\tfrac{35}{4\mathrm{i}}e^{2\l}\,\l_xI_{4b}\,v+Q_2\,w,
 $$
where
 $$
-128\,e^{-4\l}K_y^3Q_1=J_{5a}\ \text{ and }\
32\,e^{-3\l}I_3K_xQ_2=J_{5b}
 $$
are differential invariants and provided $I_{4b}=0$ (which implies $I_{5b}=0$, see identities in \S\ref{S2}) they reduce to $J_{5a}|_{I_{4b}=0}=J_{5b}|_{I_{4b}=0}=J_5$, where
 $$
J_5=5I_3(I_{5a}-I_{5c})+(I_{4a}-I_{4c})(I_{4c}-6I_{4a})-25I_2I_3^3
 $$

We see that the coefficients of $E_6$ (as well as that of other
$E_i$) are not invariant (neither are real), but the condition of
their vanishing is invariant (and real).

If $E_6$ vanishes, the system $\E'$ can be prolonged to the second jets, but is not yet formally integrable. Another curvature -- Weyl tensor $W_2'$ -- is the obstruction to prolongation to the third and henceforth infinite jets. Since $g_2'=0$, it is the Frobenius condition of the canonical Cartan distribution on the first prolongation $\E'_2$ of $\E'_1=\E'$ (but it is one equation, not three as one can guess without calculation of Spencer $\d$-cohomology!). Originally a (linear) function on $\E_2'$, it can be represented as a linear function on $\E'$ due to isomorphism $\pi_{2,1}':\E'_2\stackrel\sim\to\E'$.

This new equation $\ti E_6$ has coefficients of order 6, but they can be simplified modulo the conditions $I_{4b}=0$, $J_5=0$. Indeed we can differentiate these along invariant fields $\nabla_1$, $\nabla_2$, see for instance the next section (this allows to express all the higher invariants $I_{ij}$ with $i\ge5$ through invariants of order $\le4$). Thus the second obstruction to formal integrability $W_2'$ is the following equation, which turns out to be a linear function on $J^0(2,3)$ (we multiply it by the factor $64e^{-3\l}I_3^3K_xK_y$):
 $$
\tilde E_6=J_4\cdot(K_x^2u-K_y^2v),
 $$
where
 $$
J_4=3(I_{4a}-I_{4c})(I_{4a}+4I_{4c})I_{4c}-15I_2I_3^3(I_{4a}+4I_{4c})+25I_3^5.
 $$
Notice that in expression for $E_6$ we simplified modulo the conditions $I_{4b}=0,J_5=0$. Otherwise the coefficients are complex and more complicated, and in addition there are terms with $v_x$ and $v$. For instance the coefficient of $v_x$ term is $\frac{35}{8\mathrm{i}}e^{2\l}(I_{5b}+I_{5d})I_3^{-1}$, but it simplifies to zero.

Since formal (=local due to finite type condition) integrability of $\E'$ means existence of 4 integrals of the geodesic flow, we get the following statement:

 \begin{theorem}\label{thdim4}
The condition of exactly 4 quadratic integrals $\dim\Jj_2=4$ can be expressed as
3 differential conditions on the metric:
 $$
I_{4b}=0,\quad J_5=0,\quad J_4=0.
 $$
 \end{theorem}

%%%%%%%%%%%%
\section{\hps Digression I: Darboux-K{\oe}ning's theorem}\label{S5}

 \abz
We can deduce now the classical theorem due to Darboux and K{\oe}nning:

 \begin{theorem}
A plane metric has exactly 4 quadratic integrals iff it has
exactly one linear integral and one more quadratic integral
independent of the Hamiltonian and the square of the linear
integral.
 \end{theorem}

To one side this was proved in \cite{D}, while to the other it was given
in \cite{Koe}. It is instructive to see the equivalence by using
differential invariants only (thus below is an alternative proof of
this classical theorem):

 \smallskip

 \begin{proof}
Let us suppose at first that $g$ has 4 quadratic integrals. We may assume $K\ne\op{const}$. Thus $I_{4b}=0$ and $J_5=0$. We must show $I_{5d}=0$ (this follows easily from the condition that $v_x$-coefficient of $\tilde E_6$ vanishes, but we will show that it suffices to use only the first two conditions of Theorem \ref{thdim4}).

Note that under condition $I_{4b}=0$ we have: $\nabla_2 I_2=0$, $\nabla_2 I_3=0$ and $\nabla_2 I_{4a}=0$ (see identities of \S\ref{S2}). The latter follows from $I_{5b}=0$ as well as from the fact that the commutator $[\nabla_1,\nabla_2]$ is proportional to $\nabla_2$. Now equation $J_5=0$ can be written as
 $$
5I_3\nabla_1I_{4a}-16I_{4a}^2+2I_{4a}I_{4c}+4I_{4c}^2-20I_2I_3^3=0.
 $$
Applying $\nabla_2$ to this we get $2I_{5d}(I_{4a}+4I_{4c})=0$, which yields either $I_{5d}=0$ or $I_{4c}=-\tfrac14I_{4a}$. The latter after application of $\nabla_2$ gives $I_{5d}=0$ as well.

Now suppose that $g$ has a Killing vector field and an additional quadratic integral, so that the dimension of the space of quadratic integrals is at least 3. Since $I_{4b}=0$, the equation $E_6$ is of order 0.
If $E_6\not\equiv0$, then its prolongation adds a new first order equation to the system and the symbols satisfy: $\dim g_0\le2$, $\dim g_1=0$, so that the space of quadratic integrals cannot have dimension greater than $2$. If all the coefficients of $E_6$ vanish, then $J_5=0$.
If $J_4\ne0$, then $\tilde E_6$ is non-zero and of order 0. The same calculus for dimensions of symbols and solutions space leads to contradiction. On the other hand, if $J_4=0$, then $\tilde E_6\equiv0$
and we have 4 quadratic integrals.
 \end{proof}

 \begin{cor}
If $g$ possesses a Killing vector field, then its local degree of mobility $\dim\Jj_2$ is even: 2, 4 or 6.
 \end{cor}

%%%%%%%%%%%%
\section{\hps Digression II: On the number of invariants}\label{Sd1}

 \abz
Conditions $I_{4b}=0,J_5=0$ do not imply $J_4=0$. This pair of relations for differential invariants can be considered as an overdetermined system, but it is compatible meaning they do not produce new differential relations of lower order. Actually, we showed in the previous section that the two relations imply $I_{5d}=0$. Relations $\nabla_1I_{4b}=0$ gives $I_{5b}=0$ and $\nabla_2I_{4b}=0$ yields $I_{5c}=(I_{4a}I_{4c}-I_{4c}^2-I_2I_3^3)/I_3$. Then $J_5=0$ implies $I_{5a}=\frac25(3I_{4a}^2-I_{4a}I_{4c}-2I_{4c}^2+10I_2I_3^3)/I_3$.

Further derivations of these identities with $\nabla_s$ yield expressions for higher differential invariants $I_{ij}$, $i\ge6$, via invariants of order $\le4$ and they agree (there are 8 equations to determine 5 invariants of order 6, 12 equations to determine 6 invariants of order 7 etc), which manifests the above mentioned compatibility.

On the other hand, under certain genericity assumption, namely $I_{4c}(2I_{4a}+3I_{4c})\ne5I_2I_3^3$, the conditions $I_{4b}=0$, $J_4=0$ imply $J_5=0$. Indeed if we express $I_{5a},I_{5b},I_{5c},I_{5d}$ from $\nabla_1I_{4b}=0,\nabla_2I_{4b}=0,\nabla_1J_4=0,\nabla_2J_4=0$, and substitute this into $J_5$, the expression will have the factor $J_4$. Thus in this case the criterion of 4 integrals can be expressed as two differential conditions
 $$
I_{4b}=0,\quad J_4=0.
 $$
In general, however, we cannot remove the condition $J_5=0$ from Theorem \ref{thdim4}.\footnote{Indeed if the indicated inequality of forth order is an identity, we have 3 differential conditions of order 4 and so the condition $J_5=0$ can be reduced in order, but since this leads to an expression with roots, we do not provide it here.}

{\bf Example.} For the metric $g=\ve_1e^{(\b+2)x}dx^2+\ve_2e^{\b x}dy^2$ ($\ve_k=\pm1$; this is one family from the classification of \cite{BMM}) we have (the first two identities are obvious because $\p_y$ is the Killing field):
 \begin{multline*}
I_{4b}=0,\quad I_{5d}=0,\quad
J_5=\tfrac1{64}e^{-10(\b+2)x}\b^6(\b-1)(\b-6)(2+\b)^6,\\
J_4=\tfrac{\ve_1}{1024}e^{-15(\b+2)x}\b^{10}(\b-1)(\b+2)^9(3\b+22).
 \end{multline*}
Since $I_3=\frac{\ve_1}4e^{-3(\b+2)x}\b^2(\b+2)^2$, the cases $\b=-2,0$ correspond to constant curvature. Otherwise $J_5=0$ for $\b=1$ or $\b=6$. In the first case $J_4=0$ and we have $\dim\Jj_2(g)=4$. But in the second case $\dim\Jj_2(g)=2$.

Note also that $J_4=0$ for $\b=-22/3$, but then $I_{4c}(2I_{4a}+3I_{4c})=5I_2I_3^3$ and this does not imply $J_5=0$.

 \begin{rk}
$J_4$ is a forth order invariant obtained via reduction from a 6th order invariant modulo the conditions $I_{4b}=0$, $J_5=0$ and their $\nabla_i$-prolongations. Thus its vanishing alone without $J_5=0$ has no geometrical meaning.
 \end{rk}

%%%%%%%%%%%%
\section{\hps Precisely 3 quadratic integrals}\label{S6}

 \abz
If the compatibility condition $E_6=0$ is not trivial, then we add it and get a new system $\E''$. In this section we consider the generic case when this new equation is of order 1 in $u,v,w$, i.e. $I_{4b}\ne0$.

Then the symbol of the system $\E''$ is $g_1''=0$, i.e. it is of Frobenius type.
Its Spencer cohomology group $H^{0,2}(\E'')\simeq\R^3$, so the
obstruction to integrability -- curvature tensor -- $W_1''$ has 3
components, represented by 3 linear relations on
$J^0(2,3)$. Indeed, we can express from $\E''$ all derivatives $u_x,u_y,v_x,v_y,w_x,w_y$, calculate 3 difference of pairs of mixed derivatives and substitute the derivative expressions. We get the following equations:
 \begin{equation}\label{pant}
E_7=Au+Bw=0,\qquad E''_7=\bar Bu+\bar Aw=0,\qquad E'_7=\tfrac12(E'_7+E''_7)=0,
 \end{equation}
where $A,B$ are certain complex differential expressions of order 6 in metric (see below). One peculiarity of (\ref{pant}) is absence of $v$. Another is that
there are only two equations, not three as expected from the general theory.

Vanishing of $E_7,E''_7$ is equivalent to four real conditions $A=0,B=0$, which
can be expressed via differential invariants of order 6. In the following sections we will show that $I_{4b}=0$, but $J_5\ne0$ or $J_4\ne0$ implies $\dim\Jj_2<3$ and so we obtain: the following criterion (note that $I_{6e}$ does not enter the formulae):

 \begin{theorem}\label{thdeg3}
The condition of exactly 3 quadratic integrals is equivalent to two inequalities $I_3\ne0$, $I_{4b}\ne0$ and 4 differential relations on the metric:
  \begin{multline*}
I_{6a}=\frac1{175I_3^2I_{4b}}
 \bigl(
 700I_3^5I_{4b} - 825I_2I_3^4I_{5b} +
50I_2I_3^3I_{4b} (31I_{4a} - 18I_{4c})\\
 + 6I_{4b} (I_{4a} - I_{4c}) (6I_{4a}^2 + 49I_{4b}^2
- 37I_{4a}I_{4c} + 6I_{4c}^2 ) - 25I_3^2I_{5b} (-8I_{5a} +I_{5c})\\
 - 5I_3 (48I_{4a}^2I_{5b} - 27I_{5b}I_{4c}^2 +
2I_{4b}I_{4c} (-11I_{5a} + 46I_{5c} )\\
 + I_{4a} (-43I_{5a}I_{4b} - 21I_{5b}I_{4c} +
8I_{4b}I_{5c}) + 7I_{4b}^2 (4I_{5b} - 11I_{5d}) )
 \bigr)
 \end{multline*}
 \begin{multline*}
I_{6b}=\frac1{175I_3^2I_{4b}}
 \bigl(
1505I_2I_3^3I_{4b}^2 + 72I_{4a}^2I_{4b}^2 +
245I_3I_{5a}I_{4b}^2 + 588I_{4b}^4 + 225I_3^2I_{5b}^2\\
 + 405I_3I_{4b}I_{5b}I_{4c}
 + 72I_{4b}^2I_{4c}^2 - 6I_{4a}I_{4b} (55I_3I_{5b} +74I_{4b}I_{4c})  - 490I_3I_{4b}^2I_{5c}
 \bigr)
 \end{multline*}
  \begin{multline*}
I_{6c}=\frac1{175I_3^2I_{4b}}
 \bigl(
-175I_3^5I_{4b} + 300I_2I_3^4I_{5b} - 25I_3^2I_{5a}I_{5b} -
 100I_2I_3^3I_{4b} (5I_{4a} -9I_{4c})\\
-6I_{4b} (I_{4a} -I_{4c}) (6I_{4a}^2 +49I_{4b}^2 -37I_{4a}I_{4c} + 6I_{4c}^2) + 200I_3^2I_{5b}I_{5c} + 5I_3 (6I_{4a}^2I_{5b}\\
\quad +36I_{5b}I_{4c}^2 - I_{4b}I_{4c} (I_{5a}+34I_{5c}) +6I_{4a} (I_{5a}I_{4b} -7I_{5b}I_{4c} -I_{4b}I_{5c}) + 7I_{4b}^2(8I_{5b}- I_{5d}) )
 \bigr)
 \end{multline*}
  \begin{multline*}
I_{6d}=\frac1{175I_3^2I_{4b}}
 \bigl(
1500I_2^2I_3^6 + 36I_{4a}^4 + 25I_3^2I_{5a}^2 +
  245I_3I_{5a}I_{4b}^2 + 588I_{4b}^4 + 225I_3^2I_{5b}^2\\
- 294I_{4a}^3I_{4c} + 895I_3I_{4b}I_{5b}I_{4c} - 185I_3I_{5a}I_{4c}^2 + 366I_{4b}^2I_{4c}^2 + 36I_{4c}^4 + 6I_{4a}^2 (61I_{4b}^2\\
+ 86I_{4c}^2 -5I_3( 2I_{5a} -9I_{5c}) )- 225I_3^2I_{5a}I_{5c} -
 490I_3I_{4b}^2I_{5c} + 220I_3I_{4c}^2I_{5c} + 200I_3^2I_{5c}^2 \\
+ 5I_2I_3^3 (102I_{4a}^2 - 294I_{4a}I_{4c} + 4(49I_{4b}^2 + 48I_{4c}^2)
 +I_3(-85I_{5a} +260I_{5c}))- 245I_3I_{4b}I_{4c}I_{5d}\\
\quad - I_{4a}(6I_{4c} (172I_{4b}^2+ 49I_{4c}^2) + 5I_3(-49I_{4c}(I_{5a} - 2I_{5c}) +
 I_{4b}(164I_{5b} - 49I_{5d}) ))
 \bigr)
 \end{multline*}
 \end{theorem}

 \begin{rk}\label{rk1}
Denoting the above four equations (i.e. l.h.s-r.h.s.) by $V_1,V_2,V_3,V_4$, we can write $A=(V_2+V_4)+i(V_1+V_3)$, $B=(3V_2-V_4)+i(3V_3-V_1)$.
 \end{rk}

{\bf Example.} Consider the metric
 $$
ds^2=(x^2+q_2(y))(dx^2+dy^2),
 $$
where $q_2(y)=ay^2+by+c$. This metric is in Liouville form and hence
has an additional quadratic integral. We can calculate the
invariants from the previous theorems to find when the space of
quadratic integrals has dimension $D>2$. Here's the result according
to dimension:
 \begin{itemize}
\item[$\triangleleft$] $D=6$ if $a=1\ \&\ 4c=b^2$;
\item[$\triangleleft$] $D=4$ if $a=1\ \&\ 4c\ne b^2$;
\item[$\triangleleft$] $D=3$ if $a=4^{\pm1}\ \&\ b,c$ arbitrary.
 \end{itemize}
Note that the integrable metric $(x^2+4y^2+1)ds^2_\text{Eucl}$ was
found in classification of Matveev \cite{M$_3$}. However the methods
used by him are global and do not apply to local non-complete
situation.

 \begin{rk}
One can substitute the general Liouville form $ds^2=\Lambda\cdot(dx^2+dy^2)$ in local conformal coordinates into the above four expressions. The result is a system of 3 PDEs of order 6 in $\Lambda$ together with the equation $\Lambda_{xy}=0$ (which simplifies the 3 PDEs a lot). This system is not of finite type (for instance because it contains the cases of 4 integrals depending on 1 function of 1 variable) and it is not formally integrable: an easy elimination reduces one PDE of order 6 to order 5. Then its prolongation yields two new PDEs of order 5, but they are too long to be treated effectively.

In fact, normal forms of metrics with 2 additional integrals are better obtained with a different approach, see \cite{Koe}.
 \end{rk}

%%%%%%%%%%%%
\section{\hps Digression III: Simplification of invariants}\label{Sd2p}

 \abz
The four relations from Theorem \ref{thdeg3} provide the complete set, characterizing the condition $\dim\Jj_2=3$, but they are not compatible in the following sense. If we deduce the differential corollaries via derivations $\nabla_1,\nabla_2$, some of them will have lower order and be simpler. Let us indicate this.

Substitution of the expressions of $I_{6a},I_{6b},I_{6c},I_{6d}$ from Theorem \ref{thdeg3} to the identity (twice: before and after derivations!)
 \begin{multline*}
\nabla_1I_{6b}-\nabla_2I_{6a}\\
=(6I_{4b}I_3^3+6I_2I_{5b}I_3^2+3I_2I_{4b}(I_{4a}+I_{4c})I_3-5I_{6a}I_{4b}+4I_{4a}I_{6b}-4I_{6b} I_{4c}+3I_{4b}I_{6c})/I_3
 \end{multline*}
yields us the following new relation:
 \begin{multline*}
I_{6e}=
\frac{1}{13475I_3^2I_{4b}^3}
\bigl(375I_2^2(-34I_{4a}I_{4b}+764I_{4c}I_{4b}+75I_3I_{5b})I_3^6+61250I_{4b}^3I_3^5\\
+4500I_{5b}^3I_3^3+1125I_{5b}I_{5c}^2I_3^3+1125I_{5a}^2I_{5b}I_3^3-2250I_{5a}
I_{5b}I_{5c}I_3^3-10I_2(906I_{4b}I_{4a}^3\\
-6(225I_3I_{5b}+2068I_{4b}
I_{4c})I_{4a}^2+(6I_{4b}(917I_{4b}^2+3083I_{4c}^2)-5I_3(151I_{5a}I_{4b}+94I_{5c}I_{4b}\\
-315I_{5b}I_{4c}))I_{4a}-4I_{4b}
I_{4c}(4333I_{4b}^2+1749I_{4c}^2)+1125I_3^2I_{5b}(I_{5a}-I_{5c})-5I_3(35(22I_{5b}\\
+7I_{5d})I_{4b}^2-4I_{4c}(524I_{5a}-769I_{5c})I_{4b}+45I_{5b}I_{4c}^2))I_3^3
-450I_{5a}I_{5b}I_{4c}^2I_3^2\\
+19300I_{4b}I_{4c}I_{5c}^2I_3^2-2800I_{5a}I_{4b}^2I_{5b}I_3^2+55900I_{4b}
I_{5b}^2I_{4c}I_3^2+9500I_{5a}^2I_{4b}I_{4c}I_3^2\\
+450I_{5b}I_{4c}^2I_{5c}I_3^2-27825I_{4b}^2I_{5b}I_{5c}I_3^2
-28800I_{5a}I_{4b}I_{4c}I_{5c}I_3^2-2450I_{5a}I_{4b}^2I_{5d}I_3^2\\
-9800I_{4b}I_{5b}I_{4c}I_{5d}I_3^2+2450I_{4b}^2I_{5c}I_{5
d}I_3^2+45I_{5b}I_{4c}^4I_3-13600I_{5a}I_{4b}I_{4c}^3I_3\\
+111560I_{4b}^2I_{5b}I_{4c}^2I_3+26705I_{4b}^4I_{5b}I_3+40110I_{5a}I_{4b}^3I_{4c}I_3
+15560I_{4b}I_{4c}^3I_{5c}I_3\\
-108465I_{4b}^3I_{4c}I_{5c}I_3+45080I_{4b}^4I_{5d}I_3-15190I_{4b}^2
I_{4c}^2I_{5d}I_3+2340I_{4b}I_{4c}^5+25698I_{4b}^3I_{4c}^3\\
-1440I_{4a}^5I_{4b}+53802I_{4b}^5I_{4c}+60I_{4a}^4(27I_3I_{5b}+235I_{4b}I_{4c})
+6I_{4a}^3(-2458I_{4b}^3\\
-6625I_{4c}^2I_{4b}+10I_3(40I_{5a}I_{4b}+9I_{5c}I_{4b}
-63I_{5b}I_{4c}))+I_{4a}^2(67644I_{4c}I_{4b}^3+45300I_{4c}^3I_{4b}\\
-2700I_3^2I_{5b}(I_{5a}-I_{5c})+5I_3(2(2476I_{5b}+931I_{5d})I_{4b}^2
+2I_{4c}(3148I_{5c}-2315I_{5a})I_{4b}\\
+549I_{5b}I_{4c}^2))-2I_{4a}(25(20I_{4b}I_{5a}^2+9(I_{4b}I_{5c}-7I_{5b}I_{4c})I_{5a}
+63I_{5b}I_{4c}I_{5c}+I_{4b}(373I_{5b}^2\\
+49I_{5d}I_{5b}-29I_{5c}^2))I_3^2-5(2359I_{5c}I_{4b}^3+I_{4c}(1813I_{5d}
-16307I_{5b})I_{4b}^2-4758I_{4c}^2I_{5c}I_{4b}\\
-63I_{5b}I_{4c}^3+I_{5a}(679I_{4b}^3+3435I_{4c}^2I_{4b}))I_3+3I_{4b}(4067I_{4b}^4+15599
I_{4c}^2I_{4b}^2+3425I_{4c}^4))\bigr).
 \end{multline*}

Using similar identities for $\nabla_1I_{6c}-\nabla_2I_{6b}$, $\nabla_1I_{6d}-\nabla_2I_{6c}$, $\nabla_1I_{6e}-\nabla_2I_{6d}$ and substitutions of the 6th order invariants via the lower ones, we get 3 differential relations of order 5 (but they are non-linear even in higher order basic invariants). The first of them is:
 \begin{multline*}
1500I_2^2I_3^6-5I_2(-102I_{4a}^2+294I_{4c}I_{4a}-6(49I_{4b}^2+32I_{4c}^2)
+5I_3(17I_{5a}-52I_{5c}))I_3^3\\
+25I_{5a}^2I_3^2+275I_{5b}^2I_3^2+200I_{5c}^2I_3^2-225I_{5a}I_{5c}I_3^2
-175I_{5b}I_{5d}I_3^2+245I_{5a}I_{4b}^2I_3\\
-185I_{5a}I_{4c}^2I_3+1265I_{4b}I_{5b}I_{4c}I_3-1225I_{4b}^2I_{5c}
I_3+220I_{4c}^2I_{5c}I_3-280I_{4b}I_{4c}I_{5d}I_3\\
+36I_{4a}^4+1176I_{4b}^4+36I_{4c}^4+438I_{4b}^2I_{4c}^2-294I_{4a}^3I_{4c}
+6I_{4a}^2(73I_{4b}^2+86I_{4c}^2-5I_3(2I_{5a}\\
-9I_{5c}))-I_{4a}(6I_{4c}(246I_{4b}^2+49I_{4c}^2)+5I_3(I_{4b}(188I_{5b}
-91I_{5d})-49I_{4c}(I_{5a}-2I_{5c})))=0
 \end{multline*}
and the other two are more complicated.

Furthermore these three relations can be invariantly differentiated and
then simplified with substitutions, which resembles Cartan's prolongation-projection method, though for differential invariants. In a sequel one gets "compatible" set of relations for differential invariants, but this involves consideration of cases (lots of inequalities and equalities) and will be omitted.

%%%%%%%%%%%%
\section{\hps Generic case: Liouville form}

 \abz
Here we continue investigation of the previous section, when $I_{4b}\ne0$.
Suppose that not all equalities of the previous theorem hold.
Then $E_7$ is a non-trivial equation. If $E''_7$ is independent of it,
we get $u=w=0$ and then $v=\op{const}\cdot e^{-\l}$, so that there exists
no quadratic integral besides the Hamiltonian.

Thus for existence of an additional quadratic integral the corresponding
determinant $|A|^2-|B|^2$ should vanish (note that this implies $w=\bar u$,
which could be predicted because integral $F$ is real). In this case the
symbols dimensions are $\dim g_0=2$, $\dim g_1=0$, so for Liouville (quadratic)
integrability of the metric $g$ the system $\E'''=\{E_1=0,\dots,E_7=0\}$ should be compatible.

There are precisely two compatibility conditions: $\D_xE_7=0\,\op{mod}\E'''$
and $\D_yE_7=0\,\op{mod}\E'''$. The reduction $\op{mod}\E'''$ can be considered
here as follows: all derivatives are expressed from the first 6 equations and
substituted into derivatives of $E_7$. Then the equations are again linear and
contain only $u$- and $w$-terms. Writing linear dependence with $E_7$ we get
vanishing of two (complex) determinants. This constitutes 4 real relations
of order 7, but we write them as 2 complex relations.

In the theorem below $A,B$ are differential invariants from (\ref{pant})
(expressions are given in Remark \ref{rk1}) and
$\mathfrak{J}_1,\mathfrak{J}_2,\mathfrak{J}_3,\mathfrak{J}_4$ are some
differential invariants of order 7, precise form of which is given in Appendix.

 \begin{theorem}\label{mainTh}
Suppose that $K\ne\op{const}$, $I_{4b}\ne0$ and $|A|^2+|B|^2\ne0$ (cases considered separately).
Then the metric $g$ is Liouville iff it satisfies one real relation of order 6:
$|A|^2=|B|^2$ and 4 real relations of order 7:
 $$
B\,\mathfrak{J}_1=A\,\mathfrak{J}_2,\qquad A\,\mathfrak{J}_3=B\,\mathfrak{J}_4.
 $$
 \end{theorem}

Thus the problem of invariant characterization of Liouville metrics is solved.

 \begin{rk}
Similar to Section \ref{Sd2p} one can reduce in order and simplify differential relations
from Theorem \ref{mainTh}, but since the resulting minimal set is very cumbersome
(collection of cases involving equalities and inequalities), it won't be discussed.
 \end{rk}

%%%%%%%%%%%%
\section{\hps Singular locus: 2 quadratic integrals}\label{S6}

 \abz
Consider now the last case $I_{4b}=0$, but suppose that either
$J_5\ne0$ or $J_4\ne0$. In this case the equation $E_6$ (resp.
$\tilde E_6$) transforms into the equation (since $K\ne\op{const}$,
we may assume $K_x\ne0$ or $K_y\ne0$; formulae below are easily
adjustable to one of the cases):
 \begin{equation}\label{sl2}
K_x^2u=K_y^2w.
 \end{equation}
Prolonging this equation and using the system
$\E'=\{E_1=\dots=E_5=0\}$ we can rewrite the new system $\bar\E$
(prolongation of $\E''$) in the form:
 \begin{alignat}{3}\label{sl2a}
&u_x=\Bigl(2\log\frac{K_y}{K_x}\Bigr)_x\,u,\quad & v_x=-\l_x\,v+\Bigl(\log\frac{K_y}{K_x}-\l\Bigr)_y\,w,\quad & w_x=0,\\
&u_y=0,\quad & v_y=-\l_y\,v+\Bigl(\log\frac{K_x}{K_y}-\l\Bigr)_x\,u,\quad & w_y=\Bigl(2\log\frac{K_x}{K_y}\Bigr)_y\cdot w,\nonumber
 \end{alignat}
considered together with (\ref{sl2}). System (\ref{sl2a}) consists
of a three pair of equations, two uncoupled and one coupled with the
other two. The system is of Frobenius type. Writing compatibility
conditions of $\bar\E$ modulo (\ref{sl2})+(\ref{sl2a}) we get 3
conditions on the system to be integrable.

These three conditions are dependent (2 conditions), but modulo the
condition $I_{4b}=0$ they collapse to only one condition $I_{5d}=0$.

Note that the system has dimension of symbols $\dim\bar
g_0=2,\dim\bar g_1=0$, so that the maximal dimension of the solution
space is 2. Since the minimum is 1, we arrive to the following
statement:

 \begin{theorem}
Let $I_{4b}=0$, but either $J_5\ne0$ or $J_4\ne0$. Then the system
is Liouville iff $I_{5d}=0$ and in this case there exists only one
additional (independent of the Hamiltonian) quadratic integral.
 \end{theorem}

Note that condition $I_{4b}=0,I_{5d}=0$ are characteristic for
existence of local Killing vector field. Thus we conclude:

 \begin{cor}
Riemannian metric $g$ possesses a local Killing field iff $I_{4b}=0$
and there is a quadratic integral, independent of the Hamiltonian.
 \end{cor}

Note that if the space of such additional integrals is 1, a
representative can be chosen as the square of a linear integral.

%%%%%%%%%%%%
\section{\hps Liouville metrics: some global questions}

 \abz

 \begin{prop}
Let Liouville metric on $M^2$ have non-constant curvature and $H$ be the corresponding Hamiltonian. Then for any two quadratic integrals $F,G$ such that the triple $(F,G,H)$ is linear independent (over $\R$), the triple is functionally independent (in particular the integral $\{F,G\}$ is non-zero).
 \end{prop}

 \begin{proof}
Since $F,G,H$ are quadrics in $p$, the only kind of functional dependence for them can be either linear or quadratic.

Assume at first that the integrals $H,F,G$ are linear dependent over $C^\infty(M)$, i.e. $H=a\cdot F+b\cdot G$, where $a,b\in C^\infty(M)$ are non-constant. Then bracketing this with $H$ we get $\{H,a\}F+\{H,b\}G=0$, which would imply that $F,G,H$ have a common factor:
$H=v\cdot w$, $F=v\cdot \zeta$, $G=v\cdot \eta$. Commutation of $H$ and $F,G,H$ gives
 $$
\{H,\zeta\}=\{v,w\}\zeta,\ \{H,\eta\}=\{v,w\}\eta,\ \{H,w\}=\{v,w\}w.
 $$
Substitution of $w=a\cdot\zeta+b\cdot\eta$ into the last equality yields $\{H,a\}=\{H,b\}=0$, i.e. $a=\op{const},b=\op{const}$, so that $F,G,H$ are linearly dependent.

If $H,F,G$ are linear independent over $C^\infty(M)$, but are functionally dependent, then they must satisfy a quadratic relation:
 $$
U^2+V^2=W^2, \qquad U=\zeta^2-\eta^2,\ V=2\,\zeta\eta,\ W=\zeta^2+\eta^2,
 $$
($\zeta,\eta\in C^\infty(T^*M)$ are linear in $p$ functions) with a constant non-degenerate transition matrix $A$:
 $$
(F,G,H)=(U,V,W)\cdot A.
 $$
Denoting $\Theta=\{\zeta,\eta\}$ we observe this factor in all pair-wise Poisson brackets $\{U,V\}$, $\{U,W\}$, $\{V,W\}$, so that
 $$
0=\{H,F\} =(\a U+\b V+\g W)\cdot\Theta.
 $$
Here $\a,\b,\g$ are certain minors of the matrix $A$ and since $F$ and $H$ are non-proportional, some of them are non-zero, implying $\Theta=0$. This yields $\{H,\zeta\}=0$ and $\{H,\eta\}=0$. Thus we have two Killing vector fields and $K=\op{const}$.
 \end{proof}

This proposition immediately implies the following statement, known due to Kolokoltsev and Matveev (\cite{Kol,M$_1$,BMF}):

 \begin{cor}
The only closed Riemannian surfaces that admit more than one
additional quadratic integrals are the standard round sphere and
flat torus in the oriented case and the standard projective plane in
the non-oriented one.
 \end{cor}

 \begin{proof}
Indeed, if the metric has non-constant curvature and two additional
integrals then the Hamiltonian flow is resonant: every trajectory is
given by equation $\{H=c_1,F=c_2,G=c_3\}$ and hence is closed. Thus
$(M^2,g)$ is either $S^2$ or $\R P^2$. An additional investigation
of metrics on $S^2$ with all the geodesics closed leads to
$K=\op{const}$. The standard round sphere has 3 Killing vector
fields and thus 6 quadratic integrals. They all descend to the
standard projective plane.

Consider now the case of constant curvature. If a closed surface has
negative constant curvature, its metric is non-integrable
\cite{Koz,P}. For positive curvature we are already done. For zero
curvature we get torus or Klein bottle. Torus has 2 Killing vector
fields and the symmetric square yields 3 quadratic integrals.
However a flat Klein bottle has only one Killing vector field and
the number of quadratic integrals (including Hamiltonian) is two.
 \end{proof}

Global classification of Liouville metrics is discussed in \cite{Ki,IKS,Kol,M$_2$}. The final classification was achieved in \cite{M$_1$}, see the review \cite{BMF} for other contributions, results and references.

%%%%%%%%%%%%
\section{\hps Cubic integrals}

 \abz
Let us consider next the case of integrals of degree 3. Again for brevity sake
we bring the metric to hyperbolic conformal form $H=e^{-\l(x,y)}p_xp_y$. Then function
$F=u(x,y)p_x^3+v(x,y)p_x^2p_y+w(x,y)p_xp_y^2+\varrho(x,y)p_y^3$ is an integral of the geodesic flow if the following equations hold:
 \begin{multline*}
u_y=0,\qquad u_x+v_y+3u\l_x+v\l_y=0,\qquad v_x+w_y+2u\l_x+2v\l_y=0,\\ w_x+\varrho_y+w\l_x+3\varrho\l_y=0,\qquad \varrho_x=0.
 \end{multline*}
Denoting this system of PDEs by $\E=\{E_1=E_2=E_3=E_4=E_5=0\}$ we get criterion of integrability:
 $$
E_6=\{E_1,E_2,E_3,E_4,E_5\}\,\op{mod}(E_1,E_2,E_3,E_4,E_5)=0.
 $$
Multiplied by $\frac2{15}\,e^{\l}$ this PDE has the form $E_6=K_xu_{xx}+K_yv_{xx}+1^\text{st}\,\text{order terms}$,
and its vanishing is equivalent to $K=\op{const}$. In this and only in this case the dimension of the solution space for $\E$ is $\dim\mathcal{J}_3=10$.

If $K\ne\op{const}$, we add equation $E_6$ and get a new system $\E'=\E\cap\{E_6=0\}$, which has the following symbols: $\dim g_0=4,\dim g_1=3,\dim g_2=1,\dim g_3=0$.
Thus its solution space has dimension at most 8. The only Spencer second $\d$-cohomology groups are: $H^{1,2}(\E')\simeq\R^1$ and $H^{2,2}(\E')\simeq\R^1$. Thus the Weyl tensor has two components $W_2$ and $W_3$. The first can be obtained as follows.

Prolongation of $E_6$ to 3rd jets together with the system $\E$ yields 17 third order PDEs, while there's 16 third order differential monomials. Elimination gives the following equation of order 2:
 \begin{multline*}
E_7=-4320\,\imag e^{\l}K_x^2I_3^4I_{4b}\,v_{xx}+
 64I_3K_x^5\bigl(12(7I_{4a}-9\,\imag I_{4b}-\\
 -2I_{4c})(I_{4a}-2\,\imag I_{4b}-I_{4c})
+ 5I_3(72I_2I_3^2-14I_{5a}+29\,\imag I_{5b}+16I_{5c}-\imag I_{5d}) \bigr) \,u_x\\
 -80\,\imag e^\lambda K_x^2I_3^3\bigl((I_{5b}+I_{5d})K_x+162I_3I_{4b}\l_x\bigr)\,v_x\\
 -4e^{2\l}K_xI_3^3\bigl(12(7I_{4a}-9\,\imag I_{4b}-2I_{4c})(I_{4a}+2\,\imag I_{4b}-I_{4c})\\  +5I_3(72I_2I_3^2-14I_{5a}-29\,\imag I_{5b}+16I_{5c}+\imag I_{5d}) \bigr)\,w_x+ 0^\text{th}\,\text{order terms}=0.
 \end{multline*}

Thus vanishing of $E_7$ implies $I_{4b}=0$ and $I_{5d}=0$, so that there is a Killing vector field. Moreover further investigation of coefficients gives $J_5=0$ and
 $$
50I_3^5+5I_2I_3^3(I_{4a}+4I_{4c})-I_{4c}(I_{4a}^2+3I_{4a}I_{4c}-4I_{4c}^2)=0.
 $$
This latter condition (notice the expression is similar to $J_4$, but different) leads however to contradiction: The conditions $I_{4b}=0,I_{5d}=0,J_5=0$ allows to express all invariants of order $\ge5$ through invariants $I_2,I_3,I_{4a},I_{4c}$. Applying $\nabla_1$ to the above expression yields thus 3 polynomial equations on $I_{4a},I_{4b}$, which are compatible only with $I_3=0$. Thus we get:

 \begin{theorem}
If a metric $g$ has non-constant curvature, then $\dim\Jj_3\le7$.
 \end{theorem}

In fact, we can continue and consider the system $\E''=\E'\cap\{E_7=0\}$. It has symbols with $\dim g_0=4,\dim g_1=3,\dim g_2=0$. The non-zero second Spencer $\d$-cohomology group is $H^{1,2}(\E'')\simeq\R^3$ and the compatibility is the Frobenius condition on the second jets, which leads (by vanishing of coefficients) to a complicated overdetermined polynomial system on differential invariants of order $\le5$ (higher order are expressed via these). Computer investigation indicates incompatibility, implying strict inequality in the above theorem.

In a similar way we can continue prolongation-projection method for 6 more times. Finally we arrive to high order (and highly non-linear in $I_{\z\t}$) differential invariants, which express existence of at least one cubic integral. Intermediate steps give more invariants, describing super-integrable cases, but it is rather complicated to decide what is the precise number of the conditions (because there are relations via derivations $\nabla_1,\nabla_2$).

Moreover it seems that $\dim\Jj_3$ can be neither 6 nor 5, i.e. the next realized dimension of $\Jj_3$ after 10 is 4! However this has no proof so far (as well as the fact that this implies 4 quadratic integrals). To the reverse side we have:

 \begin{theorem}
If $\dim\Jj_2=4$, then $\dim\Jj_3=4$.
 \end{theorem}

 \begin{proof}
We will exploit the following statement, which can be derived from the works of V.Matveev:
 \begin{lem}\label{rrr}
If two metrics $g$, $\bar g$ are projectively equivalent, then for any $k\ge1$: $\dim\Jj_k(g)=\dim\Jj_k(\bar g)$.
 \end{lem}
Actually the statement holds for any $n=\dim M$ (to a certain extent this can be found in \cite{TM} for $k=2$, but the case of general $k$ is similar): A local diffeomorphism $\vp:(M,g)\to(\overline{M},\bar g)$ is a projective transformation iff the map
 $$
F\mapsto \bar F=(\op{det}G)^{-\frac{k}{n+1}}\cdot\vp_*(F)
 $$
is the isomorphism $\mathcal{J}_k(g)\simeq\mathcal{J}_k(\bar g)$ $\forall k$. Here $\vp_*(F):=(\vp^{-1})^*(F)\circ\vp$, $\vp_*g$ is defined similarly and
 $$
G=\sharp^{\bar g}\circ\flat_{\vp_*g}:T\overline{M}\to T\overline{M},
 $$
where $\flat_g:TM\to T^*M$, $\sharp^{\bar g}:T^*\overline{M}\to T\overline{M}$ are the natural morphisms of shifting indices.

 \begin{rk}
Denote $\mathcal{G}(g)$ the space of metrics geodesically equivalent to $g$ ($\vp=\op{Id}$ above). Then according to {\rm\cite{TM}} $\mathcal{G}(g)\simeq\mathcal{J}_2(g)$ with the equivalence being given by
 $$
\bar g\mapsto I=(\op{det}G)^{\frac2{n+1}}\cdot\bar g.
 $$
 \end{rk}

By the results of \cite{BMM} (now again $n=2$) any (pseudo-) Riemannian metric $g$ with $\dim\Jj_2(g)=4$ is projectively equivalent to a metric of the family
 $$
g_0=e^{3x}\,dx^2+\sigma e^x\,dy^2\ \simeq\ x\cdot ds^2_0,\qquad \z\ne0.
 $$
Here $ds^2_0$ is the standard Euclidean or Minkovsky metric on $\R^2(x,y)$.

By the above remark and Lemma \ref{rrr} it is enough to investigate $\Jj_3(g)$ for $g=g_0$ only. Since the latter representative for $g_0$ has the simplest form, this is an easy investigation and the result is $\dim\Jj_3(g_0)=4$.
 \end{proof}

%%%%%%%%%%%%
\section{\hps Higher order integrals}

 \abz
When we pass to integrals of degree $n$, as we noted at the beginning, the system $\E$ is given by $(n+2)$ equations on $(n+1)$ unknowns. This system is of finite type and is a generalized complete intersection, so by theorem C of \cite{KL$_3$} its solution space has dimension $\frac12(n+1)(n+2)$ iff the compatibility condition $K=\op{const}$ holds. Otherwise the dimension drops at least by 2:

 \begin{prop}
If $\dim\!\Jj_n\ge\frac{n^2+3n}2$, then the inequality is strict and $K=\op{const}$.\!
 \end{prop}

 \begin{proof}
Indeed, since the equations in the system $\E$ are linear of the first order $E_i=(\nabla^1_i+\nabla^0_i)({\bf u})$ and have constant coefficient of first order terms $\nabla^1_i$, order of the multi-bracket $E_{n+3}=\{E_1,\dots,E_{n+2}\}$ drops (compared to expected $(n+1)$ in general) and becomes $n$ in pure form and $(n-1)$ after reduction by equations $E_i$. So if the new equation is not zero, the symbols of the new system satisfy: $\dim g'_i\le n+1-i$ for $i<n-1$ and $\dim g'_{n-1}=1$, $\dim g'_n=0$. Then $\dim\op{Sol}(\E)\le\sum\dim g_i'<\frac{n^2+3n}2$.
 \end{proof}

Further steps of prolongation-projection generalize Darboux-K\oe ning theorem, but are more complicated. To understand this let us give more details on the integrals for the metric $g_0=x\cdot(dx^2+dy^2)$ from the previous section.

The Killing form is $K=q$ (with $p,q$ being the momenta dual to $x,y$) and the 4 quadratic integrals are:
 $$
H=\frac{p^2+q^2}{x},\quad K^2=q^2,\quad F=yH-2p\,K,\quad G=y^2H-4(yp-xq)K.
 $$
The latter integral can be considered as the additional integral from the K\oe ning's theorem because $\{K,G\}=2F$, but $\{K,F\}=H$ (with a slight difference in generators, these relations were also observed in \cite{KKW}).

For cubic integrals we have: $\mathcal{J}_3(g_0)=\mathcal{J}_1(g_0)\cdot\mathcal{J}_2(g_0)=\langle HK,K^3,KF,KG\rangle$. In fact, Poisson brackets of $\mathcal{J}_2(g_0)$ give nothing new: $\{G,F\}=16K^3$.

For integrals of higher degree we have: $\mathcal{J}_k(g_0)=\mathcal{J}_2(g_0)\cdot\mathcal{J}_{k-2}(g_0)$ for $k>2$. There are however relations, which are generated by precisely 1 relation in degree 4: $HG-F^2=4K^4$.
Thus for $d_k(g)=\dim\mathcal{J}_k(g)$ we have:
 \begin{equation}\label{qaz}
d_{2k}(g_0)=d_{2k+1}(g_0)=\dim S^k\mathcal{J}_2(g_0)-\dim S^{k-2}\mathcal{J}_2(g_0).
 \end{equation}
This implies:
 \begin{theorem}
If $d_2(g)=4$, then $d_{2k}(g)=d_{2k+1}(g)=(k+1)^2$.
 \end{theorem}

 \begin{proof}
Indeed, by the same argument as in the last proof $d_2(g)=4$ implies
$d_n(g)=d_n(g_0)$ and the latter quantity for $n=2k$ or $2k+1$ due
to formula (\ref{qaz}) equals $\binom{k+3}3-\binom{k+1}3=(k+1)^2$.
 \end{proof}

It is natural to expect that the cases from the last theorem are the next in prolongation-projection method after the space forms:

\begin{quote}
{\bf Conjecture.} If $K\ne\op{const}$, then $d_n(g)\le ([\frac n2]+1)^2$ and the equality is attained for metrics with $d_2(g)=4$.
\end{quote}

For $n=2$ this obviously holds, for $n=3$ we supported this by arguments in the previous section, while for $n\ge4$ this seems to be hardly treated via successive prolongation-projection scheme.

One is tempted to suggest a kind of monotonicity as an approach, i.e. that $d_2(g)<d_2(h)$ could imply $d_n(g)<d_n(h)$ for two metrics $g,h$, but this would be wrong. For instance there are Liouville metrics with $d_2(g)=2,d_3(g)=0$, but there are other metrics, for which the cubic integrals are the simplest polynomial integrals: $d_2(h)=1,d_3(h)=1$. Indeed, according to \cite{Te} there are metrics $g_k$ such that $d_{2k+1}(g_k)\ge1$, while for $i\le k$: $d_{2i-1}(g_k)=0$, $d_{2i}(g_k)=1$ (the latter is nonzero because Hamiltonian is always an integral).

This was proved in the loc.sit. paper via a simple calculation, but it also follows from our approach, because the criterion of existence of non-trivial integrals of degree $n$ (i.e. $\mathcal{J}_n\ne0$ for odd $n$ and $d_n>1$ for even $n$) is given by a criterion via differential invariants of order, which is monotonic in $n$. In particular, we can arrange $d_{2k+1}(g_k)=1$ for the above sequence.

%%%%%%%%%
\appendix

\section{Long formulae}

Below are the expressions for the seventh order differential
invariants involved in Theorem \ref{mainTh}. The calculations are
performed using \textsc{Mathematica}\footnote{Copy of the notebook
with detailed computations is available from the author.}.

\newpage

\hspace{-60pt}\vbox{\footnotesize
 %{\tiny
 \begin{multline*}
\mathfrak{J}_1=
 142500{{I_2}}^3{{I_3}}^9 + 216{{I_{4a}}}^6 - 125{{I_3}}^3{{I_{5a}}}^3 +
  6125\imag {{I_3}}^6{I_{5a}}{I_{4b}} - 875\imag {{I_3}}^3{I_{5a}}{I_{6a}}{I_{4b}} +
  350{{I_3}}^2{{I_{5a}}}^2{{I_{4b}}}^2 \\
  - 6125{{I_3}}^3{I_{7a}}{{I_{4b}}}^2 -
  133525\imag {{I_3}}^5{{I_{4b}}}^3 - 25725\imag {{I_3}}^2{I_{6a}}{{I_{4b}}}^3 +
  23030{I_3}{I_{5a}}{{I_{4b}}}^4 + 98784{{I_{4b}}}^6
  + 1125\imag {{I_3}}^3{{I_{5a}}}^2{I_{5b}} \\
  - 73500{{I_3}}^6{I_{4b}}{I_{5b}} +
  14000{{I_3}}^3{I_{6a}}{I_{4b}}{I_{5b}} +
  48825\imag {{I_3}}^2{I_{5a}}{{I_{4b}}}^2{I_{5b}} -
  12250\imag {I_3}{{I_{4b}}}^4{I_{5b}}
  - 16250{{I_3}}^3{I_{5a}}{{I_{5b}}}^2 \\
+ 51275{{I_3}}^2{{I_{4b}}}^2{{I_{5b}}}^2 + 36000\imag {{I_3}}^3{{I_{5b}}}^3 +
  5250{{I_3}}^3{I_{5a}}{I_{4b}}{I_{6b}} - 26950{{I_3}}^2{{I_{4b}}}^3{I_{6b}}
  - 45500\imag {{I_3}}^3{I_{4b}}{I_{5b}}{I_{6b}} \\
+ 12250\imag {{I_3}}^3{{I_{4b}}}^2{I_{7b}} +
  36{{I_{4a}}}^5( -5\imag {I_{4b}} - 126{I_{4c}} )  -
  1350\imag {{I_3}}^2{{I_{5a}}}^2{I_{4b}}{I_{4c}}
  - 67375{{I_3}}^5{{I_{4b}}}^2{I_{4c}} +
  2800{{I_3}}^2{I_{6a}}{{I_{4b}}}^2{I_{4c}} \\
+ 51905\imag {I_3}{I_{5a}}{{I_{4b}}}^3{I_{4c}} + 76440\imag {{I_{4b}}}^5{I_{4c}} -
  24350{{I_3}}^2{I_{5a}}{I_{4b}}{I_{5b}}{I_{4c}}
  + 171185{I_3}{{I_{4b}}}^3{I_{5b}}{I_{4c}} +
  85300\imag {{I_3}}^2{I_{4b}}{{I_{5b}}}^2{I_{4c}} \\
- 42175\imag {{I_3}}^2{{I_{4b}}}^2{I_{6b}}{I_{4c}} +
  2700{{I_3}}^2{{I_{5a}}}^2{{I_{4c}}}^2
  - 25725\imag {{I_3}}^5{I_{4b}}{{I_{4c}}}^2 +
  175\imag {{I_3}}^2{I_{6a}}{I_{4b}}{{I_{4c}}}^2 -
  22490{I_3}{I_{5a}}{{I_{4b}}}^2{{I_{4c}}}^2 \\
+ 120288{{I_{4b}}}^4{{I_{4c}}}^2 -
  3950\imag {{I_3}}^2{I_{5a}}{I_{5b}}{{I_{4c}}}^2
  + 90105\imag {I_3}{{I_{4b}}}^2{I_{5b}}{{I_{4c}}}^2 +
  6750{{I_3}}^2{{I_{5b}}}^2{{I_{4c}}}^2 + 11200{{I_3}}^2{I_{4b}}{I_{6b}}{{I_{4c}}}^2\\
+  13070\imag {I_3}{I_{5a}}{I_{4b}}{{I_{4c}}}^3
  + 28470\imag {{I_{4b}}}^3{{I_{4c}}}^3 +
  71440{I_3}{I_{4b}}{I_{5b}}{{I_{4c}}}^3 - 13315{I_3}{I_{5a}}{{I_{4c}}}^4 +
  34122{{I_{4b}}}^2{{I_{4c}}}^4 + 25245\imag {I_3}{I_{5b}}{{I_{4c}}}^4 \\
  + 2340\imag {I_{4b}}{{I_{4c}}}^5 + 2556{{I_{4c}}}^6 +
  6{{I_{4a}}}^4( 942{{I_{4b}}}^2 + 635\imag {I_{4b}}{I_{4c}} + 4715{{I_{4c}}}^2 -
  90{I_3}( {I_{5a}} - 3\imag {I_{5b}} - 8{I_{5c}} )  )
  + 3000{{I_3}}^3{{I_{5a}}}^2{I_{5c}} \\
- 30625\imag {{I_3}}^6{I_{4b}}{I_{5c}} +
  875\imag {{I_3}}^3{I_{6a}}{I_{4b}}{I_{5c}} +
  11550{{I_3}}^2{I_{5a}}{{I_{4b}}}^2{I_{5c}} - 112210{I_3}{{I_{4b}}}^4{I_{5c}}
  - 5750\imag {{I_3}}^3{I_{5a}}{I_{5b}}{I_{5c}} \\
- 61075\imag {{I_3}}^2{{I_{4b}}}^2{I_{5b}}{I_{5c}} + 19750{{I_3}}^3{{I_{5b}}}^2{I_{5c}} +
  7000{{I_3}}^3{I_{4b}}{I_{6b}}{I_{5c}}
  + 16700\imag {{I_3}}^2{I_{5a}}{I_{4b}}{I_{4c}}{I_{5c}} -
  120505\imag {I_3}{{I_{4b}}}^3{I_{4c}}{I_{5c}} \\
+ 96800{{I_3}}^2{I_{4b}}{I_{5b}}{I_{4c}}{I_{5c}} -
  30950{{I_3}}^2{I_{5a}}{{I_{4c}}}^2{I_{5c}}
  + 1315{I_3}{{I_{4b}}}^2{{I_{4c}}}^2{I_{5c}} +
  53650\imag {{I_3}}^2{I_{5b}}{{I_{4c}}}^2{I_{5c}} -
  35470\imag {I_3}{I_{4b}}{{I_{4c}}}^3{I_{5c}} \\
+ 18320{I_3}{{I_{4c}}}^4{I_{5c}} -
  17875{{I_3}}^3{I_{5a}}{{I_{5c}}}^2
  - 11900{{I_3}}^2{{I_{4b}}}^2{{I_{5c}}}^2 +
  29125\imag {{I_3}}^3{I_{5b}}{{I_{5c}}}^2 -
  39850\imag {{I_3}}^2{I_{4b}}{I_{4c}}{{I_{5c}}}^2 \\
+ 30700{{I_3}}^2{{I_{4c}}}^2{{I_{5c}}}^2 + 15000{{I_3}}^3{{I_{5c}}}^3
  + 25{{I_2}}^2{{I_3}}^6( 2298{{I_{4a}}}^2 - 6790{{I_{4b}}}^2 +
  6{I_{4a}}( -785\imag {I_{4b}} - 1701{I_{4c}} )  +
  3900\imag {I_{4b}}{I_{4c}} \\
+ 7908{{I_{4c}}}^2 +  5{I_3}( -383{I_{5a}}
  + 561\imag {I_{5b}} + 1888{I_{5c}} )  )  +
  2625\imag {{I_3}}^3{I_{5a}}{I_{4b}}{I_{6c}} +
  26950\imag {{I_3}}^2{{I_{4b}}}^3{I_{6c}} - 17500{{I_3}}^3{I_{4b}}{I_{5b}}{I_{6c}}\\
-  62300{{I_3}}^2{{I_{4b}}}^2{I_{4c}}{I_{6c}} -
  25025\imag {{I_3}}^2{I_{4b}}{{I_{4c}}}^2{I_{6c}} -
  27125\imag {{I_3}}^3{I_{4b}}{I_{5c}}{I_{6c}} -
  6\imag {{I_{4a}}}^3( 365{{I_{4b}}}^3 - 11213\imag {{I_{4b}}}^2{I_{4c}}
  + 3615{I_{4b}}{{I_{4c}}}^2 \\
- 10395\imag {{I_{4c}}}^3 +
     5{I_3}( 3{I_{5a}}( 13{I_{4b}} + 84\imag {I_{4c}} )  +
        14{I_{4c}}( 19{I_{5b}} - 109\imag {I_{5c}} )  +
        {I_{4b}}( -678\imag {I_{5b}}
  + 31{I_{5c}} + 392\imag {I_{5d}} ) ) ) \\
- 6125\imag {{I_3}}^2{I_{5a}}{{I_{4b}}}^2{I_{5d}} +
  17150\imag {I_3}{{I_{4b}}}^4{I_{5d}} - 51450{{I_3}}^2{{I_{4b}}}^2{I_{5b}}{I_{5d}}
  + 22050{{I_3}}^2{I_{5a}}{I_{4b}}{I_{4c}}{I_{5d}} -
  62965{I_3}{{I_{4b}}}^3{I_{4c}}{I_{5d}} \\
- 39200\imag {{I_3}}^2{I_{4b}}{I_{5b}}{I_{4c}}{I_{5d}} +
  7350\imag {I_3}{{I_{4b}}}^2{{I_{4c}}}^2{I_{5d}}
  - 21560{I_3}{I_{4b}}{{I_{4c}}}^3{I_{5d}} +
  18375\imag {{I_3}}^2{{I_{4b}}}^2{I_{5c}}{I_{5d}} -
  39200{{I_3}}^2{I_{4b}}{I_{4c}}{I_{5c}}{I_{5d}} \\
+ 8575{{I_3}}^2{{I_{4b}}}^2{{I_{5d}}}^2
  + 8750{{I_3}}^3{I_{5a}}{I_{4b}}{I_{6d}} +
  31850{{I_3}}^2{{I_{4b}}}^3{I_{6d}} - 14000\imag {{I_3}}^3{I_{4b}}{I_{5b}}{I_{6d}} +
  37625\imag {{I_3}}^2{{I_{4b}}}^2{I_{4c}}{I_{6d}}\\
  - 14000{{I_3}}^2{I_{4b}}{{I_{4c}}}^2{I_{6d}} - 21000{{I_3}}^3{I_{4b}}{I_{5c}}{I_{6d}} +
  5{I_2}{{I_3}}^3( 1296{{I_{4a}}}^4 - 11025\imag {{I_3}}^5{I_{4b}} + 33614{{I_{4b}}}^4
  + 6{{I_{4a}}}^3( -575\imag {I_{4b}} \\
- 2534{I_{4c}} )  +
   29995\imag {{I_{4b}}}^3{I_{4c}} + 43420{{I_{4b}}}^2{{I_{4c}}}^2 +
   12480\imag {I_{4b}}{{I_{4c}}}^3 + 14316{{I_{4c}}}^4 +
  2{{I_{4a}}}^2( 12295{{I_{4b}}}^2 + 10665\imag {I_{4b}}{I_{4c}} \\
+ 20418{{I_{4c}}}^2 - 30{I_3}( 36{I_{5a}} - 59\imag {I_{5b}} - 239{I_{5c}} )
   )  - {I_{4a}}( 17535\imag {{I_{4b}}}^3
  + 58410{{I_{4b}}}^2{I_{4c}} +
    30360\imag {I_{4b}}{{I_{4c}}}^2 + 41244{{I_{4c}}}^3 \\
+ 10{I_3}( {I_{5a}}( -165\imag {I_{4b}} - 1267{I_{4c}} )  +
    7{I_{4c}}( 249\imag {I_{5b}} + 734{I_{5c}} )
   + 3{I_{4b}}( 1818{I_{5b}} + 440\imag {I_{5c}} - 441{I_{5d}} )  )
        )  + 5{I_3}( {I_{5a}} ( 4375{{I_{4b}}}^2 \\
+ 470\imag {I_{4b}}{I_{4c}} - 2102{{I_{4c}}}^2 )  +
        926{{I_{4c}}}^2( 3\imag {I_{5b}}
   + 8{I_{5c}} )  +
        2{I_{4b}}{I_{4c}}( 5984{I_{5b}} + 325\imag {I_{5c}} - 2548{I_{5d}} )
   +35\imag {{I_{4b}}}^2( 115{I_{5b}} + 139\imag {I_{5c}} \\
- 49{I_{5d}} ))  + 25{{I_3}}^2( 36{{I_{5a}}}^2
   + 35\imag {I_{6a}}{I_{4b}} +
        678{{I_{5b}}}^2 - 112{I_{4b}}{I_{6b}} +
        {I_{5a}}( -118\imag {I_{5b}} - 478{I_{5c}} )  +
        650\imag {I_{5b}}{I_{5c}} + 932{{I_{5c}}}^2 \\
- 301\imag {I_{4b}}{I_{6c}}
  -448{I_{4b}}{I_{6d}} )  )  +
  {{I_{4a}}}^2( -7350\imag {{I_3}}^5{I_{4b}} +
     6( 7308{{I_{4b}}}^4 + 5975\imag {{I_{4b}}}^3{I_{4c}} +
        31787{{I_{4b}}}^2{{I_{4c}}}^2 + 5835\imag {I_{4b}}{{I_{4c}}}^3\\
+ 9915{{I_{4c}}}^4 )  + 5{I_3}( {I_{5a}}
         ( 1277{{I_{4b}}}^2 + 32\imag {I_{4b}}{I_{4c}} - 5363{{I_{4c}}}^2 )  +
        2423{{I_{4c}}}^2( 3\imag {I_{5b}} + 8{I_{5c}} )  +
        2{I_{4b}}{I_{4c}}( 12162{I_{5b}}
   + 1489\imag {I_{5c}} \\
- 5733{I_{5d}} )  +
        \imag {{I_{4b}}}^2( 5176{I_{5b}} + 1102\imag {I_{5c}} - 1715{I_{5d}} )
        )  + 150{{I_3}}^2( 3{{I_{5a}}}^2 + 7\imag {I_{6a}}{I_{4b}} +
        130{{I_{5b}}}^2 - 42{I_{4b}}{I_{6b}} + {I_{5a}}( -18\imag {I_{5b}} \\
- 48{I_{5c}} )  + 46\imag {I_{5b}}{I_{5c}} +
        143{{I_{5c}}}^2 - 21\imag {I_{4b}}{I_{6c}} - 70{I_{4b}}{I_{6d}} )  )  +
  12250\imag {{I_3}}^3{{I_{4b}}}^2{I_{7d}}
  - 6125\imag {{I_3}}^2{{I_{4b}}}^3{I_{6e}} \\
+ 14700{{I_3}}^2{{I_{4b}}}^2{I_{4c}}{I_{6e}} +
  {I_{4a}}( 1225{{I_3}}^5{I_{4b}}( 130{I_{4b}} + 27\imag {I_{4c}} )  -
     6\imag ( 980{{I_{4b}}}^5
  - 35756\imag {{I_{4b}}}^4{I_{4c}} +  16355{{I_{4b}}}^3{{I_{4c}}}^2 \\
- 27203\imag {{I_{4b}}}^2{{I_{4c}}}^3 +
        3215{I_{4b}}{{I_{4c}}}^4 - 3941\imag {{I_{4c}}}^5 )  +
     5{I_3}( {I_{5a}}( -3731\imag {{I_{4b}}}^3
  - 379{{I_{4b}}}^2{I_{4c}} -
           2412\imag {I_{4b}}{{I_{4c}}}^2 + 6622{{I_{4c}}}^3 ) \\
+ 14{{I_{4c}}}^3( -789\imag {I_{5b}} - 1054{I_{5c}} )  -
        7{{I_{4b}}}^3( 2361{I_{5b}} - 1513\imag {I_{5c}}
  - 1449{I_{5d}} )  +
        {{I_{4b}}}^2{I_{4c}}( -26097\imag {I_{5b}} + 239{I_{5c}} +
           5145\imag {I_{5d}} )  \\
+ 2{I_{4b}}{{I_{4c}}}^2( -17272{I_{5b}} + 2151\imag {I_{5c}} + 6713{I_{5d}})  )
  + 25{{I_3}}^2( -2832\imag {I_{4b}}{{I_{5b}}}^2 +
        1547\imag {{I_{4b}}}^2{I_{6b}} +
        {{I_{5a}}}^2( 44\imag {I_{4b}} - 126{I_{4c}} ) \\
 + 7{I_{6a}}{I_{4b}}( 9{I_{4b}} - 7\imag {I_{4c}} )  -
        1050{{I_{5b}}}^2{I_{4c}}
  - 196{I_{4b}}{I_{6b}}{I_{4c}} -
        2092{I_{4b}}{I_{5b}}{I_{5c}} - 2422\imag {I_{5b}}{I_{4c}}{I_{5c}} -
        26\imag {I_{4b}}{{I_{5c}}}^2 \\
- 2086{I_{4c}}{{I_{5c}}}^2 + 1232{{I_{4b}}}^2{I_{6c}}
  + 1127\imag {I_{4b}}{I_{4c}}{I_{6c}} +
        588\imag {I_{4b}}{I_{5b}}{I_{5d}} + 1078{I_{4b}}{I_{5c}}{I_{5d}} -
        2{I_{5a}}( 7{I_{4c}}( -19\imag {I_{5b}} - 109{I_{5c}} )  \\
+ {I_{4b}}( 53{I_{5b}} + 9\imag {I_{5c}} + 196{I_{5d}} )  )  +
        35\imag {{I_{4b}}}^2{I_{6d}} + 980{I_{4b}}{I_{4c}}{I_{6d}} -
        343{{I_{4b}}}^2{I_{6e}} )  )  + 6125{{I_3}}^3{{I_{4b}}}^2{I_{7e}} .
 \end{multline*}
}

\hspace{-60pt}\vbox{\footnotesize
 %{\tiny
 \begin{multline*}
\mathfrak{J}_2=
 142500{{I_2}}^3{{I_3}}^9 + 216{{I_{4a}}}^6 - 125{{I_3}}^3{{I_{5a}}}^3 +
  2625\imag {{I_3}}^6{I_{5a}}{I_{4b}} - 875\imag {{I_3}}^3{I_{5a}}{I_{6a}}{I_{4b}} -
  4550{{I_3}}^2{{I_{5a}}}^2{{I_{4b}}}^2 \\
 - 6125{{I_3}}^3{I_{7a}}{{I_{4b}}}^2 +
  126175\imag {{I_3}}^5{{I_{4b}}}^3 + 13475\imag {{I_3}}^2{I_{6a}}{{I_{4b}}}^3 -
  23030{I_3}{I_{5a}}{{I_{4b}}}^4 - 98784{{I_{4b}}}^6 +
  625\imag {{I_3}}^3{{I_{5a}}}^2{I_{5b}} \\
 - 66500{{I_3}}^6{I_{4b}}{I_{5b}} +
  14000{{I_3}}^3{I_{6a}}{I_{4b}}{I_{5b}} -
  32375\imag {{I_3}}^2{I_{5a}}{{I_{4b}}}^2{I_{5b}} -
  58310\imag {I_3}{{I_{4b}}}^4{I_{5b}} - 19750{{I_3}}^3{I_{5a}}{{I_{5b}}}^2 \\
 - 1225{{I_3}}^2{{I_{4b}}}^2{{I_{5b}}}^2 - 36000\imag {{I_3}}^3{{I_{5b}}}^3 +
  8750{{I_3}}^3{I_{5a}}{I_{4b}}{I_{6b}} - 2450{{I_3}}^2{{I_{4b}}}^3{I_{6b}} +
  42000\imag {{I_3}}^3{I_{4b}}{I_{5b}}{I_{6b}} \\
 - 12250\imag {{I_3}}^3{{I_{4b}}}^2{I_{7b}} +
  36{{I_{4a}}}^5( 19\imag {I_{4b}} - 126{I_{4c}} )  -
  1450\imag {{I_3}}^2{{I_{5a}}}^2{I_{4b}}{I_{4c}} - 12075{{I_3}}^5{{I_{4b}}}^2{I_{4c}} +  2800{{I_3}}^2{I_{6a}}{{I_{4b}}}^2{I_{4c}} \\
 - 14455\imag {I_3}{I_{5a}}{{I_{4b}}}^3{I_{4c}} - 175224\imag {{I_{4b}}}^5{I_{4c}} -
  24350{{I_3}}^2{I_{5a}}{I_{4b}}{I_{5b}}{I_{4c}} -
  58275{I_3}{{I_{4b}}}^3{I_{5b}}{I_{4c}} -
  95100\imag {{I_3}}^2{I_{4b}}{{I_{5b}}}^2{I_{4c}} \\
 +14525\imag {{I_3}}^2{{I_{4b}}}^2{I_{6b}}{I_{4c}} +
  2700{{I_3}}^2{{I_{5a}}}^2{{I_{4c}}}^2 + 23975\imag {{I_3}}^5{I_{4b}}{{I_{4c}}}^2 +
  175\imag {{I_3}}^2{I_{6a}}{I_{4b}}{{I_{4c}}}^2 -
  21370{I_3}{I_{5a}}{{I_{4b}}}^2{{I_{4c}}}^2 \\
 + 5544{{I_{4b}}}^4{{I_{4c}}}^2 +  6750\imag {{I_3}}^2{I_{5a}}{I_{5b}}{{I_{4c}}}^2 -
  176695\imag {I_3}{{I_{4b}}}^2{I_{5b}}{{I_{4c}}}^2 +
  450{{I_3}}^2{{I_{5b}}}^2{{I_{4c}}}^2 + 10500{{I_3}}^2{I_{4b}}{I_{6b}}{{I_{4c}}}^2\\
 +  15210\imag {I_3}{I_{5a}}{I_{4b}}{{I_{4c}}}^3 - 85842\imag {{I_{4b}}}^3{{I_{4c}}}^3 +  14740{I_3}{I_{4b}}{I_{5b}}{{I_{4c}}}^3 - 13315{I_3}{I_{5a}}{{I_{4c}}}^4 +
  23034{{I_{4b}}}^2{{I_{4c}}}^4 - 25875\imag {I_3}{I_{5b}}{{I_{4c}}}^4 \\
 - 7884\imag {I_{4b}}{{I_{4c}}}^5 + 2556{{I_{4c}}}^6 +
  25{{I_2}}^2{{I_3}}^6( 2298{{I_{4a}}}^2 - 770{{I_{4b}}}^2 +
     {I_{4a}}( 2890\imag {I_{4b}} - 10206{I_{4c}} )  -
     9780\imag {I_{4b}}{I_{4c}} + 7908{{I_{4c}}}^2 \\
 -5{I_3}( 383{I_{5a}} + 351\imag {I_{5b}} - 1888{I_{5c}} )  )  +
  6{{I_{4a}}}^4( 774{{I_{4b}}}^2 - 2245\imag {I_{4b}}{I_{4c}} + 4715{{I_{4c}}}^2 -
     30{I_3}( 3{I_{5a}} - 5\imag {I_{5b}} - 24{I_{5c}} )  )  \\
 +3000{{I_3}}^3{{I_{5a}}}^2{I_{5c}} + 21875\imag {{I_3}}^6{I_{4b}}{I_{5c}} +
  875\imag {{I_3}}^3{I_{6a}}{I_{4b}}{I_{5c}} +
  26250{{I_3}}^2{I_{5a}}{{I_{4b}}}^2{I_{5c}} + 112210{I_3}{{I_{4b}}}^4{I_{5c}}\\
 +5750\imag {{I_3}}^3{I_{5a}}{I_{5b}}{I_{5c}} +
  44625\imag {{I_3}}^2{{I_{4b}}}^2{I_{5b}}{I_{5c}} + 16250{{I_3}}^3{{I_{5b}}}^2{I_{5c}} +
  3500{{I_3}}^3{I_{4b}}{I_{6b}}{I_{5c}} +
  14800\imag {{I_3}}^2{I_{5a}}{I_{4b}}{I_{4c}}{I_{5c}}\\
 + 101675\imag {I_3}{{I_{4b}}}^3{I_{4c}}{I_{5c}} +
  48500{{I_3}}^2{I_{4b}}{I_{5b}}{I_{4c}}{I_{5c}} -
  30950{{I_3}}^2{I_{5a}}{{I_{4c}}}^2{I_{5c}} +
  54515{I_3}{{I_{4b}}}^2{{I_{4c}}}^2{I_{5c}} \\
- 57150\imag {{I_3}}^2{I_{5b}}{{I_{4c}}}^2{I_{5c}} +
  2010\imag {I_3}{I_{4b}}{{I_{4c}}}^3{I_{5c}} + 18320{I_3}{{I_{4c}}}^4{I_{5c}} -
  17875{{I_3}}^3{I_{5a}}{{I_{5c}}}^2 - 21700{{I_3}}^2{{I_{4b}}}^2{{I_{5c}}}^2 \\
- 30875\imag {{I_3}}^3{I_{5b}}{{I_{5c}}}^2 +
  11150\imag {{I_3}}^2{I_{4b}}{I_{4c}}{{I_{5c}}}^2 +
  30700{{I_3}}^2{{I_{4c}}}^2{{I_{5c}}}^2 + 15000{{I_3}}^3{{I_{5c}}}^3 -
  875\imag {{I_3}}^3{I_{5a}}{I_{4b}}{I_{6c}} \\
+ 2450\imag {{I_3}}^2{{I_{4b}}}^3{I_{6c}} - 10500{{I_3}}^3{I_{4b}}{I_{5b}}{I_{6c}} - 7000{{I_3}}^2{{I_{4b}}}^2{I_{4c}}{I_{6c}} +
  24675\imag {{I_3}}^2{I_{4b}}{{I_{4c}}}^2{I_{6c}} +
  25375\imag {{I_3}}^3{I_{4b}}{I_{5c}}{I_{6c}} \\
- 6{{I_{4a}}}^3( -3067\imag {{I_{4b}}}^3 + 8161{{I_{4b}}}^2{I_{4c}} -
     12365\imag {I_{4b}}{{I_{4c}}}^2 + 10395{{I_{4c}}}^3 +
     5{I_3}( {I_{5a}}( 87\imag {I_{4b}} - 252{I_{4c}} )  +
        14{I_{4c}}( -15\imag {I_{5b}} \\
+ 109{I_{5c}} )+ {I_{4b}}( 818{I_{5b}} - 353\imag {I_{5c}} - 392{I_{5d}} )  )  )  -
  6825\imag {{I_3}}^2{I_{5a}}{{I_{4b}}}^2{I_{5d}} +
  53410\imag {I_3}{{I_{4b}}}^4{I_{5d}} - 15750{{I_3}}^2{{I_{4b}}}^2{I_{5b}}{I_{5d}}\\
 + 22050{{I_3}}^2{I_{5a}}{I_{4b}}{I_{4c}}{I_{5d}} -
  45045{I_3}{{I_{4b}}}^3{I_{4c}}{I_{5d}} +
  44100\imag {{I_3}}^2{I_{4b}}{I_{5b}}{I_{4c}}{I_{5d}} +
  33950\imag {I_3}{{I_{4b}}}^2{{I_{4c}}}^2{I_{5d}}\\
 - 21560{I_3}{I_{4b}}{{I_{4c}}}^3{I_{5d}} -
  5425\imag {{I_3}}^2{{I_{4b}}}^2{I_{5c}}{I_{5d}} -
  39200{{I_3}}^2{I_{4b}}{I_{4c}}{I_{5c}}{I_{5d}} +
  8575{{I_3}}^2{{I_{4b}}}^2{{I_{5d}}}^2 + 8750{{I_3}}^3{I_{5a}}{I_{4b}}{I_{6d}} \\
 - 2450{{I_3}}^2{{I_{4b}}}^3{I_{6d}} + 17500\imag {{I_3}}^3{I_{4b}}{I_{5b}}{I_{6d}} -
  9975\imag {{I_3}}^2{{I_{4b}}}^2{I_{4c}}{I_{6d}} -
  14000{{I_3}}^2{I_{4b}}{{I_{4c}}}^2{I_{6d}} - 21000{{I_3}}^3{I_{4b}}{I_{5c}}{I_{6d}}\\
 + 5{I_2}{{I_3}}^3( 1296{{I_{4a}}}^4 + 2275\imag {{I_3}}^5{I_{4b}} - 15974{{I_{4b}}}^4 + 6\imag {{I_{4a}}}^3( 281{I_{4b}} + 2534\imag {I_{4c}} )  -
  44681\imag {{I_{4b}}}^3{I_{4c}} - 10816{{I_{4b}}}^2{{I_{4c}}}^2 \\
 - 41376\imag {I_{4b}}{{I_{4c}}}^3 + 14316{{I_{4c}}}^4 +
  2{{I_{4a}}}^2( 12827{{I_{4b}}}^2 - 16699\imag {I_{4b}}{I_{4c}} +
  20418{{I_{4c}}}^2 - 30{I_3}( 36{I_{5a}} + 3\imag {I_{5b}} - 239{I_{5c}} ) )\\
 + 5{I_3}( {I_{5a}} ( 4151{{I_{4b}}}^2 + 1266\imag {I_{4b}}{I_{4c}} - 2102{{I_{4c}}}^2 )  + 2{{I_{4c}}}^2( -1683\imag {I_{5b}} + 3704{I_{5c}} )  +
        2{I_{4b}}{I_{4c}}( 4178{I_{5b}} - 3783\imag {I_{5c}} - 2548{I_{5d}} )\\
 - 7\imag {{I_{4b}}}^2( 985{I_{5b}} - 803\imag {I_{5c}} - 615{I_{5d}} )
        )  + {I_{4a}}( 53221\imag {{I_{4b}}}^3 - 30438{{I_{4b}}}^2{I_{4c}} +
        73088\imag {I_{4b}}{{I_{4c}}}^2 - 41244{{I_{4c}}}^3\\
 + 10{I_3}( {I_{5a}}( -263\imag {I_{4b}} + 1267{I_{4c}} )  +
           7{I_{4c}}( 243\imag {I_{5b}} - 734{I_{5c}} )  +
           {I_{4b}}( -5258{I_{5b}} + 1908\imag {I_{5c}} + 1323{I_{5d}} )  )
        )  + 25{{I_3}}^2( 36{{I_{5a}}}^2 \\
+ 35\imag {I_{6a}}{I_{4b}} +
        762{{I_{5b}}}^2 - 252{I_{4b}}{I_{6b}} +
        {I_{5a}}( 6\imag {I_{5b}} - 478{I_{5c}} )  - 678\imag {I_{5b}}{I_{5c}} +
        932{{I_{5c}}}^2 + 231\imag {I_{4b}}{I_{6c}} - 448{I_{4b}}{I_{6d}} )  ) \\
 + {{I_{4a}}}^2( -3150\imag {{I_3}}^5{I_{4b}} +
     6( 1904{{I_{4b}}}^4 - 29341\imag {{I_{4b}}}^3{I_{4c}} +
        17339{{I_{4b}}}^2{{I_{4c}}}^2 - 19765\imag {I_{4b}}{{I_{4c}}}^3 + 9915{{I_{4c}}}^4)\\
 + 5{I_3}( {I_{5a}}( 2621{{I_{4b}}}^2 + 2936\imag {I_{4b}}{I_{4c}} - 5363{{I_{4c}}}^2 )  + {{I_{4c}}}^2( -8235\imag {I_{5b}} + 19384{I_{5c}} )  +
        2{I_{4b}}{I_{4c}}( 10482{I_{5b}} - 6767\imag {I_{5c}} \\
 - 5733{I_{5d}} )  + {{I_{4b}}}^2( -14024\imag {I_{5b}} - 3622{I_{5c}} + 6685\imag {I_{5d}} ))  + 150{{I_3}}^2( 3{{I_{5a}}}^2 + 7\imag {I_{6a}}{I_{4b}} +
        158{{I_{5b}}}^2 - 70{I_{4b}}{I_{6b}} +
        {I_{5a}}( -10\imag {I_{5b}} -\\
 48{I_{5c}} )  - 46\imag {I_{5b}}{I_{5c}} +
        143{{I_{5c}}}^2 + 7\imag {I_{4b}}{I_{6c}} - 70{I_{4b}}{I_{6d}} )  )  -
  12250\imag {{I_3}}^3{{I_{4b}}}^2{I_{7d}} - 11025\imag {{I_3}}^2{{I_{4b}}}^3{I_{6e}} + 14700{{I_3}}^2{{I_{4b}}}^2{I_{4c}}{I_{6e}} \\
+{I_{4a}}( 175{{I_3}}^5{I_{4b}}( 834{I_{4b}} - 119\imag {I_{4c}} )  +
     6( 17444\imag {{I_{4b}}}^5 + 5572{{I_{4b}}}^4{I_{4c}} +
        46581\imag {{I_{4b}}}^3{{I_{4c}}}^2 - 13791{{I_{4b}}}^2{{I_{4c}}}^3 \\
+  10845\imag {I_{4b}}{{I_{4c}}}^4 - 3941{{I_{4c}}}^5 )  +
     5{I_3}( {I_{5a}}( 6321\imag {{I_{4b}}}^3 - 1947{{I_{4b}}}^2{I_{4c}} -
           5456\imag {I_{4b}}{{I_{4c}}}^2 + 6622{{I_{4c}}}^3 )  +
        14{{I_{4c}}}^3( 855\imag {I_{5b}} \\
- 1054{I_{5c}} )  -
        2{I_{4b}}{{I_{4c}}}^2( 9502{I_{5b}} - 5507\imag {I_{5c}} - 6713{I_{5d}} )
            + 3{{I_{4b}}}^2{I_{4c}}( 17421\imag {I_{5b}} + 173{I_{5c}} -
           6125\imag {I_{5d}} )  \\
+ 7{{I_{4b}}}^3( 575{I_{5b}} - 2415\imag {I_{5c}} + 197{I_{5d}} )  )  +
     25{{I_3}}^2( 3224\imag {I_{4b}}{{I_{5b}}}^2 - 1281\imag {{I_{4b}}}^2{I_{6b}} +
        {{I_{5a}}}^2( 68\imag {I_{4b}} - 126{I_{4c}} ) \\
 +7{I_{6a}}{I_{4b}}( 9{I_{4b}} - 7\imag {I_{4c}} )  -
        966{{I_{5b}}}^2{I_{4c}} - 1700{I_{4b}}{I_{5b}}{I_{5c}} +
        2562\imag {I_{5b}}{I_{4c}}{I_{5c}} + 334\imag {I_{4b}}{{I_{5c}}}^2 -
        2086{I_{4c}}{{I_{5c}}}^2 + 700{{I_{4b}}}^2{I_{6c}} \\
 - 1029\imag {I_{4b}}{I_{4c}}{I_{6c}} +
        2{I_{5a}}( 7{I_{4c}}( -15\imag {I_{5b}} + 109{I_{5c}} )  +
           {I_{4b}}( 17{I_{5b}} - 201\imag {I_{5c}} - 196{I_{5d}} )  )  -
        784\imag {I_{4b}}{I_{5b}}{I_{5d}} + 1078{I_{4b}}{I_{5c}}{I_{5d}}\\
- 301\imag {{I_{4b}}}^2{I_{6d}} + 980{I_{4b}}{I_{4c}}{I_{6d}} -
        343{{I_{4b}}}^2{I_{6e}} )  )  + 6125{{I_3}}^3{{I_{4b}}}^2{I_{7e}} .
 \end{multline*}
}

\hspace{-60pt}\vbox{\footnotesize
 %{\tiny
 \begin{multline*}
\mathfrak{J}_3=
142500{{I_2}}^3{{I_3}}^9 + 216{{I_{4a}}}^6 - 125{{I_3}}^3{{I_{5a}}}^3 -
  30625\imag {{I_3}}^6{I_{5a}}{I_{4b}} + 2625\imag {{I_3}}^3{I_{5a}}{I_{6a}}{I_{4b}} +  4550{{I_3}}^2{{I_{5a}}}^2{{I_{4b}}}^2\\
 + 6125{{I_3}}^3{I_{7a}}{{I_{4b}}}^2 -
  341775\imag {{I_3}}^5{{I_{4b}}}^3 - 13475\imag {{I_3}}^2{I_{6a}}{{I_{4b}}}^3 -
  24990{I_3}{I_{5a}}{{I_{4b}}}^4 - 98784{{I_{4b}}}^6 -
  3375\imag {{I_3}}^3{{I_{5a}}}^2{I_{5b}}\\
  + 147000{{I_3}}^6{I_{4b}}{I_{5b}} -
  17500{{I_3}}^3{I_{6a}}{I_{4b}}{I_{5b}} +
  29225\imag {{I_3}}^2{I_{5a}}{{I_{4b}}}^2{I_{5b}} -
  108290\imag {I_3}{{I_{4b}}}^4{I_{5b}} + 24250{{I_3}}^3{I_{5a}}{{I_{5b}}}^2\\
  - 31675{{I_3}}^2{{I_{4b}}}^2{{I_{5b}}}^2 + 36000\imag {{I_3}}^3{{I_{5b}}}^3 -
  14000{{I_3}}^3{I_{5a}}{I_{4b}}{I_{6b}} + 46550{{I_3}}^2{{I_{4b}}}^3{I_{6b}} -
  64750\imag {{I_3}}^3{I_{4b}}{I_{5b}}{I_{6b}}\\
  + 24500\imag {{I_3}}^3{{I_{4b}}}^2{I_{7b}} +
  36\imag {{I_{4a}}}^5( 29{I_{4b}} + 126\imag {I_{4c}} )  +
  12450\imag {{I_3}}^2{{I_{5a}}}^2{I_{4b}}{I_{4c}} + 314825{{I_3}}^5{{I_{4b}}}^2{I_{4c}} -
  15750{{I_3}}^2{I_{6a}}{{I_{4b}}}^2{I_{4c}} \\
  + 91105\imag {I_3}{I_{5a}}{{I_{4b}}}^3{I_{4c}} - 22344\imag {{I_{4b}}}^5{I_{4c}} -
  23200{{I_3}}^2{I_{5a}}{I_{4b}}{I_{5b}}{I_{4c}} -
  116305{I_3}{{I_{4b}}}^3{I_{5b}}{I_{4c}} +
  8300\imag {{I_3}}^2{I_{4b}}{{I_{5b}}}^2{I_{4c}} \\
  - 43575\imag {{I_3}}^2{{I_{4b}}}^2{I_{6b}}{I_{4c}} +
  2700{{I_3}}^2{{I_{5a}}}^2{{I_{4c}}}^2 + 91875\imag {{I_3}}^5{I_{4b}}{{I_{4c}}}^2 -
  12775\imag {{I_3}}^2{I_{6a}}{I_{4b}}{{I_{4c}}}^2 -
  67700{I_3}{I_{5a}}{{I_{4b}}}^2{{I_{4c}}}^2 \\
  - 147336{{I_{4b}}}^4{{I_{4c}}}^2 +
  36350\imag {{I_3}}^2{I_{5a}}{I_{5b}}{{I_{4c}}}^2 -
  99945\imag {I_3}{{I_{4b}}}^2{I_{5b}}{{I_{4c}}}^2 -
  64350{{I_3}}^2{{I_{5b}}}^2{{I_{4c}}}^2 + 39550{{I_3}}^2{I_{4b}}{I_{6b}}{{I_{4c}}}^2\\
  + 18680\imag {I_3}{I_{5a}}{I_{4b}}{{I_{4c}}}^3 - 105342\imag {{I_{4b}}}^3{{I_{4c}}}^3 -
  55910{I_3}{I_{4b}}{I_{5b}}{{I_{4c}}}^3 - 13315{I_3}{I_{5a}}{{I_{4c}}}^4 +
  4314{{I_{4b}}}^2{{I_{4c}}}^4 - 31635\imag {I_3}{I_{5b}}{{I_{4c}}}^4\\
  - 12564\imag {I_{4b}}{{I_{4c}}}^5 + 2556{{I_{4c}}}^6 -
  25{{I_2}}^2{{I_3}}^6( -2298{{I_{4a}}}^2 +
     6{I_{4a}}( -3475\imag {I_{4b}} + 1701{I_{4c}} )  +
     4( 4270{{I_{4b}}}^2 + 5235\imag {I_{4b}}{I_{4c}}\\
  - 1977{{I_{4c}}}^2 )  +
     5{I_3}( 383{I_{5a}} + 2271\imag {I_{5b}} - 1888{I_{5c}} )  )  -
  6{{I_{4a}}}^4( -534{{I_{4b}}}^2 + 3515\imag {I_{4b}}{I_{4c}} - 4715{{I_{4c}}}^2 +
     90{I_3}( {I_{5a}} + 9\imag {I_{5b}} - 8{I_{5c}} )  )  \\
  + 3000{{I_3}}^3{{I_{5a}}}^2{I_{5c}} + 116375\imag {{I_3}}^6{I_{4b}}{I_{5c}} -
  14875\imag {{I_3}}^3{I_{6a}}{I_{4b}}{I_{5c}} -
  26250{{I_3}}^2{I_{5a}}{{I_{4b}}}^2{I_{5c}} + 148470{I_3}{{I_{4b}}}^4{I_{5c}} \\
  + 41750\imag {{I_3}}^3{I_{5a}}{I_{5b}}{I_{5c}} -
  63525\imag {{I_3}}^2{{I_{4b}}}^2{I_{5b}}{I_{5c}} - 83750{{I_3}}^3{{I_{5b}}}^2{I_{5c}} +
  50750{{I_3}}^3{I_{4b}}{I_{6b}}{I_{5c}} -
  20350\imag {{I_3}}^2{I_{5a}}{I_{4b}}{I_{4c}}{I_{5c}} \\
 - 32795\imag {I_3}{{I_{4b}}}^3{I_{4c}}{I_{5c}} +
  21450{{I_3}}^2{I_{4b}}{I_{5b}}{I_{4c}}{I_{5c}} -
  30950{{I_3}}^2{I_{5a}}{{I_{4c}}}^2{I_{5c}} +
  95455{I_3}{{I_{4b}}}^2{{I_{4c}}}^2{I_{5c}} -
  92350\imag {{I_3}}^2{I_{5b}}{{I_{4c}}}^2{I_{5c}} \\
  -15670\imag {I_3}{I_{4b}}{{I_{4c}}}^3{I_{5c}} + 18320{I_3}{{I_{4c}}}^4{I_{5c}} -
  17875{{I_3}}^3{I_{5a}}{{I_{5c}}}^2 + 21700{{I_3}}^2{{I_{4b}}}^2{{I_{5c}}}^2 -
  62875\imag {{I_3}}^3{I_{5b}}{{I_{5c}}}^2 \\
  + 20150\imag {{I_3}}^2{I_{4b}}{I_{4c}}{{I_{5c}}}^2 +
  30700{{I_3}}^2{{I_{4c}}}^2{{I_{5c}}}^2 + 15000{{I_3}}^3{{I_{5c}}}^3 -
  20125\imag {{I_3}}^3{I_{5a}}{I_{4b}}{I_{6c}} +
  75950\imag {{I_3}}^2{{I_{4b}}}^3{I_{6c}} \\
  + 77000{{I_3}}^3{I_{4b}}{I_{5b}}{I_{6c}} +
  25200{{I_3}}^2{{I_{4b}}}^2{I_{4c}}{I_{6c}} +
  40775\imag {{I_3}}^2{I_{4b}}{{I_{4c}}}^2{I_{6c}} +
  56875\imag {{I_3}}^3{I_{4b}}{I_{5c}}{I_{6c}} - 36750{{I_3}}^3{{I_{4b}}}^2{I_{7c}}\\
  - 6125\imag {{I_3}}^2{I_{5a}}{{I_{4b}}}^2{I_{5d}} +
  65170\imag {I_3}{{I_{4b}}}^4{I_{5d}} + 29400{{I_3}}^2{{I_{4b}}}^2{I_{5b}}{I_{5d}} +
  22050{{I_3}}^2{I_{5a}}{I_{4b}}{I_{4c}}{I_{5d}} +
  7595{I_3}{{I_{4b}}}^3{I_{4c}}{I_{5d}} \\
  + 61250\imag {{I_3}}^2{I_{4b}}{I_{5b}}{I_{4c}}{I_{5d}} +
  63210\imag {I_3}{{I_{4b}}}^2{{I_{4c}}}^2{I_{5d}} -
  21560{I_3}{I_{4b}}{{I_{4c}}}^3{I_{5d}} +
  6125\imag {{I_3}}^2{{I_{4b}}}^2{I_{5c}}{I_{5d}} -
  39200{{I_3}}^2{I_{4b}}{I_{4c}}{I_{5c}}{I_{5d}} \\
  + 8575{{I_3}}^2{{I_{4b}}}^2{{I_{5d}}}^2 +
  6{{I_{4a}}}^3( 3317\imag {{I_{4b}}}^3 - 3321{{I_{4b}}}^2{I_{4c}} +
     19595\imag {I_{4b}}{{I_{4c}}}^2 - 10395{{I_{4c}}}^3 +
     5{I_3}( {I_{5a}}( 89\imag {I_{4b}} + 252{I_{4c}} )\\
 +14{I_{4c}}( 127\imag {I_{5b}} - 109{I_{5c}} )  +
        {I_{4b}}( 522{I_{5b}} + 219\imag {I_{5c}} + 392{I_{5d}} )  )  )  +
  8750{{I_3}}^3{I_{5a}}{I_{4b}}{I_{6d}} - 66150{{I_3}}^2{{I_{4b}}}^3{I_{6d}}\\
 + 29750\imag {{I_3}}^3{I_{4b}}{I_{5b}}{I_{6d}} -
  17325\imag {{I_3}}^2{{I_{4b}}}^2{I_{4c}}{I_{6d}} -
  14000{{I_3}}^2{I_{4b}}{{I_{4c}}}^2{I_{6d}} - 21000{{I_3}}^3{I_{4b}}{I_{5c}}{I_{6d}} +5{I_2}{{I_3}}^3( 1296{{I_{4a}}}^4 \\
  + 47775\imag {{I_3}}^5{I_{4b}} - 15974{{I_{4b}}}^4 +
     6\imag {{I_{4a}}}^3( 2943{I_{4b}} + 2534\imag {I_{4c}} )  +
     42007\imag {{I_{4b}}}^3{I_{4c}} - 102974{{I_{4b}}}^2{{I_{4c}}}^2 -
     67008\imag {I_{4b}}{{I_{4c}}}^3 \\
+ 14316{{I_{4c}}}^4 -2{{I_{4a}}}^2( 22927{{I_{4b}}}^2 + 47787\imag {I_{4b}}{I_{4c}} -
        20418{{I_{4c}}}^2 + 30{I_3}( 36{I_{5a}} + 275\imag {I_{5b}} - 239{I_{5c}} )
        )  - 5{I_3}( {I_{5a}} ( 5481{{I_{4b}}}^2 \\
- 32\imag {I_{4b}}{I_{4c}} + 2102{{I_{4c}}}^2 )  +
        2{{I_{4c}}}^2( 4755\imag {I_{5b}} - 3704{I_{5c}} )  +
        7\imag {{I_{4b}}}^2( 131{I_{5b}} + 503\imag {I_{5c}} - 1155{I_{5d}} )  +
        2{I_{4b}}{I_{4c}}( 6955{I_{5b}} \\
+ 4391\imag {I_{5c}} + 2548{I_{5d}} )
        )  + {I_{4a}}( -21007\imag {{I_{4b}}}^3 + 142428{{I_{4b}}}^2{I_{4c}} +
        144924\imag {I_{4b}}{{I_{4c}}}^2 - 41244{{I_{4c}}}^3 +
        10{I_3}( {I_{5a}}( -761\imag {I_{4b}} \\
+ 1267{I_{4c}} )  +
           7{I_{4c}}( 915\imag {I_{5b}} - 734{I_{5c}} )  +
           {I_{4b}}( 7290{I_{5b}} + 5206\imag {I_{5c}} + 1323{I_{5d}} )  ))  + 25{{I_3}}^2( 36{{I_{5a}}}^2 - 203\imag {I_{6a}}{I_{4b}} -
        1446{{I_{5b}}}^2 \\
+ 854{I_{4b}}{I_{6b}} +
        {I_{5a}}( 550\imag {I_{5b}} - 478{I_{5c}} )  -
        2342\imag {I_{5b}}{I_{5c}} + 932{{I_{5c}}}^2 + 1099\imag {I_{4b}}{I_{6c}} -
        448{I_{4b}}{I_{6d}} )  )  +
  {{I_{4a}}}^2( 36750\imag {{I_3}}^5{I_{4b}} \\
- 6( 56{{I_{4b}}}^4 + 32091\imag {{I_{4b}}}^3{I_{4c}} +
        6741{{I_{4b}}}^2{{I_{4c}}}^2 + 31435\imag {I_{4b}}{{I_{4c}}}^3 - 9915{{I_{4c}}}^4
        )  - 5{I_3}( {I_{5a}}
         ( 5105{{I_{4b}}}^2 + 3932\imag {I_{4b}}{I_{4c}} + 5363{{I_{4c}}}^2 )\\
+{{I_{4c}}}^2( 24747\imag {I_{5b}} - 19384{I_{5c}} )  +
        \imag {{I_{4b}}}^2( 3364{I_{5b}} + 8206\imag {I_{5c}} - 9947{I_{5d}} )  +
        2{I_{4b}}{I_{4c}}( 9648{I_{5b}} + 5881\imag {I_{5c}} + 5733{I_{5d}} )) \\
 + 150{{I_3}}^2( 3{{I_{5a}}}^2 - 21\imag {I_{6a}}{I_{4b}} -
        194{{I_{5b}}}^2 + 112{I_{4b}}{I_{6b}} +
        6\imag {I_{5a}}( 9{I_{5b}} + 8\imag {I_{5c}} )  -
        334\imag {I_{5b}}{I_{5c}} + 143{{I_{5c}}}^2 + 161\imag {I_{4b}}{I_{6c}}\\
 - 70{I_{4b}}{I_{6d}} )  )  - 24500\imag {{I_3}}^3{{I_{4b}}}^2{I_{7d}} -
  23275\imag {{I_3}}^2{{I_{4b}}}^3{I_{6e}} + 14700{{I_3}}^2{{I_{4b}}}^2{I_{4c}}{I_{6e}} +
  {I_{4a}}( -1225{{I_3}}^5{I_{4b}}( 262{I_{4b}} + 105\imag {I_{4c}} ) \\
+ 6( 15484\imag {{I_{4b}}}^5 + 33012{{I_{4b}}}^4{I_{4c}} +
        40331\imag {{I_{4b}}}^3{{I_{4c}}}^2 + 8809{{I_{4b}}}^2{{I_{4c}}}^3 +
        17275\imag {I_{4b}}{{I_{4c}}}^4 - 3941{{I_{4c}}}^5 )  +
     5{I_3}( {I_{5a}}( -4711\imag {{I_{4b}}}^3 \\
 + 25945{{I_{4b}}}^2{I_{4c}} - 338\imag {I_{4b}}{{I_{4c}}}^2 + 6622{{I_{4c}}}^3 )  +
        14{{I_{4c}}}^3( 1527\imag {I_{5b}} - 1054{I_{5c}} )  +
        \imag {{I_{4b}}}^2{I_{4c}}( 25353{I_{5b}} + 26897\imag {I_{5c}} -
           27489{I_{5d}} ) \\
 + 7{{I_{4b}}}^3
         ( 93{I_{5b}} - 13\imag {I_{5c}} - 567{I_{5d}} )  +
        2{I_{4b}}{{I_{4c}}}^2( 13673{I_{5b}} + 6791\imag {I_{5c}} + 6713{I_{5d}} )
            )  + 25{{I_3}}^2( -732\imag {I_{4b}}{{I_{5b}}}^2 +
        903\imag {{I_{4b}}}^2{I_{6b}} \\
+ {{I_{5a}}}^2( -118\imag {I_{4b}} - 126{I_{4c}} )  +
        7{I_{6a}}{I_{4b}}( 15{I_{4b}} + 91\imag {I_{4c}} )  +
        3738{{I_{5b}}}^2{I_{4c}} - 2254{I_{4b}}{I_{6b}}{I_{4c}} +
        452{I_{4b}}{I_{5b}}{I_{5c}} + 5698\imag {I_{5b}}{I_{4c}}{I_{5c}}\\
- 6\imag {I_{4b}}{{I_{5c}}}^2 - 2086{I_{4c}}{{I_{5c}}}^2 -
        1148{{I_{4b}}}^2{I_{6c}} - 2597\imag {I_{4b}}{I_{4c}}{I_{6c}} +
        2{I_{5a}}( 7{I_{4c}}( -127\imag {I_{5b}} + 109{I_{5c}} )  +
           {I_{4b}}( 229{I_{5b}} + 307\imag {I_{5c}} \\
- 196{I_{5d}} )  )  - 1470\imag {I_{4b}}{I_{5b}}{I_{5d}} + 1078{I_{4b}}{I_{5c}}{I_{5d}} -
        7\imag {{I_{4b}}}^2{I_{6d}} + 980{I_{4b}}{I_{4c}}{I_{6d}} -
        343{{I_{4b}}}^2{I_{6e}} )  )  + 6125{{I_3}}^3{{I_{4b}}}^2{I_{7e}} .
 \end{multline*}
}

\hspace{-60pt}\vbox{\footnotesize
 %{\tiny
 \begin{multline*}
\mathfrak{J}_4=
 142500{{I_2}}^3{{I_3}}^9 + 216{{I_{4a}}}^6 - 125{{I_3}}^3{{I_{5a}}}^3 -
  27125\imag {{I_3}}^6{I_{5a}}{I_{4b}} + 2625\imag {{I_3}}^3{I_{5a}}{I_{6a}}{I_{4b}} -  350{{I_3}}^2{{I_{5a}}}^2{{I_{4b}}}^2 \\
 + 6125{{I_3}}^3{I_{7a}}{{I_{4b}}}^2 +
  15925\imag {{I_3}}^5{{I_{4b}}}^3 + 25725\imag {{I_3}}^2{I_{6a}}{{I_{4b}}}^3 +
  24990{I_3}{I_{5a}}{{I_{4b}}}^4 + 98784{{I_{4b}}}^6 -
  2875\imag {{I_3}}^3{{I_{5a}}}^2{I_{5b}} \\
 - 7000{{I_3}}^6{I_{4b}}{I_{5b}} -
  10500{{I_3}}^3{I_{6a}}{I_{4b}}{I_{5b}} -
  63175\imag {{I_3}}^2{I_{5a}}{{I_{4b}}}^2{I_{5b}} +
  37730\imag {I_3}{{I_{4b}}}^4{I_{5b}} + 11750{{I_3}}^3{I_{5a}}{{I_{5b}}}^2 \\
 +50225{{I_3}}^2{{I_{4b}}}^2{{I_{5b}}}^2 - 36000\imag {{I_3}}^3{{I_{5b}}}^3 -
  10500{{I_3}}^3{I_{5a}}{I_{4b}}{I_{6b}} - 17150{{I_3}}^2{{I_{4b}}}^3{I_{6b}} +
  36750\imag {{I_3}}^3{I_{4b}}{I_{5b}}{I_{6b}}\\
 + 36\imag {{I_{4a}}}^5( 5{I_{4b}} + 126\imag {I_{4c}} )  +
  12550\imag {{I_3}}^2{{I_{5a}}}^2{I_{4b}}{I_{4c}} - 44275{{I_3}}^5{{I_{4b}}}^2{I_{4c}} -
  14350{{I_3}}^2{I_{6a}}{{I_{4b}}}^2{I_{4c}} -
  55615\imag {I_3}{I_{5a}}{{I_{4b}}}^3{I_{4c}} \\
 - 76440\imag {{I_{4b}}}^5{I_{4c}} +
  19400{{I_3}}^2{I_{5a}}{I_{4b}}{I_{5b}}{I_{4c}} +
  169995{I_3}{{I_{4b}}}^3{I_{5b}}{I_{4c}} -
  86700\imag {{I_3}}^2{I_{4b}}{{I_{5b}}}^2{I_{4c}} +
  64925\imag {{I_3}}^2{{I_{4b}}}^2{I_{6b}}{I_{4c}} \\
 + 2700{{I_3}}^2{{I_{5a}}}^2{{I_{4c}}}^2 + 42175\imag {{I_3}}^5{I_{4b}}{{I_{4c}}}^2 -
  12775\imag {{I_3}}^2{I_{6a}}{I_{4b}}{{I_{4c}}}^2 +
  10900{I_3}{I_{5a}}{{I_{4b}}}^2{{I_{4c}}}^2 + 120288{{I_{4b}}}^4{{I_{4c}}}^2\\
 + 25650\imag {{I_3}}^2{I_{5a}}{I_{5b}}{{I_{4c}}}^2 -
  69665\imag {I_3}{{I_{4b}}}^2{I_{5b}}{{I_{4c}}}^2 +
  57150{{I_3}}^2{{I_{5b}}}^2{{I_{4c}}}^2 - 10150{{I_3}}^2{I_{4b}}{I_{6b}}{{I_{4c}}}^2+
  16540\imag {I_3}{I_{5a}}{I_{4b}}{{I_{4c}}}^3 \\
- 28470\imag {{I_{4b}}}^3{{I_{4c}}}^3 +
  117430{I_3}{I_{4b}}{I_{5b}}{{I_{4c}}}^3 - 13315{I_3}{I_{5a}}{{I_{4c}}}^4 +
  34122{{I_{4b}}}^2{{I_{4c}}}^4 + 19485\imag {I_3}{I_{5b}}{{I_{4c}}}^4 -
  2340\imag {I_{4b}}{{I_{4c}}}^5 \\
+ 2556{{I_{4c}}}^6 -
  25{{I_2}}^2{{I_3}}^6( -2298{{I_{4a}}}^2 +
     2{I_{4a}}( -6625\imag {I_{4b}} + 5103{I_{4c}} )  +
     12( 1085{{I_{4b}}}^2 + 605\imag {I_{4b}}{I_{4c}} - 659{{I_{4c}}}^2 ) \\
+ 5{I_3}( 383{I_{5a}} + 1359\imag {I_{5b}} - 1888{I_{5c}} )  )  -
  6{{I_{4a}}}^4( -942{{I_{4b}}}^2 + 635\imag {I_{4b}}{I_{4c}} - 4715{{I_{4c}}}^2 +
     30{I_3}( 3{I_{5a}} + 23\imag {I_{5b}} - 24{I_{5c}} )  ) \\
+ 3000{{I_3}}^3{{I_{5a}}}^2{I_{5c}} + 63875\imag {{I_3}}^6{I_{4b}}{I_{5c}} -
  14875\imag {{I_3}}^3{I_{6a}}{I_{4b}}{I_{5c}} -
  11550{{I_3}}^2{I_{5a}}{{I_{4b}}}^2{I_{5c}} - 148470{I_3}{{I_{4b}}}^4{I_{5c}}\\
+ 30250\imag {{I_3}}^3{I_{5a}}{I_{5b}}{I_{5c}} +
  72975\imag {{I_3}}^2{{I_{4b}}}^2{I_{5b}}{I_{5c}} + 47750{{I_3}}^3{{I_{5b}}}^2{I_{5c}} -
  1750{{I_3}}^3{I_{4b}}{I_{6b}}{I_{5c}} -
  18450\imag {{I_3}}^2{I_{5a}}{I_{4b}}{I_{4c}}{I_{5c}} \\
+ 50225\imag {I_3}{{I_{4b}}}^3{I_{4c}}{I_{5c}} +
  22950{{I_3}}^2{I_{4b}}{I_{5b}}{I_{4c}}{I_{5c}} -
  30950{{I_3}}^2{I_{5a}}{{I_{4c}}}^2{I_{5c}} -
  89825{I_3}{{I_{4b}}}^2{{I_{4c}}}^2{I_{5c}} +
  18450\imag {{I_3}}^2{I_{5b}}{{I_{4c}}}^2{I_{5c}} \\
- 53150\imag {I_3}{I_{4b}}{{I_{4c}}}^3{I_{5c}} + 18320{I_3}{{I_{4c}}}^4{I_{5c}} -
  17875{{I_3}}^3{I_{5a}}{{I_{5c}}}^2 + 11900{{I_3}}^2{{I_{4b}}}^2{{I_{5c}}}^2 -
  2875\imag {{I_3}}^3{I_{5b}}{{I_{5c}}}^2 \\
- 30850\imag {{I_3}}^2{I_{4b}}{I_{4c}}{{I_{5c}}}^2 +
  30700{{I_3}}^2{{I_{4c}}}^2{{I_{5c}}}^2 + 15000{{I_3}}^3{{I_{5c}}}^3 -
  16625\imag {{I_3}}^3{I_{5a}}{I_{4b}}{I_{6c}} -
  46550\imag {{I_3}}^2{{I_{4b}}}^3{I_{6c}} \\
- 49000{{I_3}}^3{I_{4b}}{I_{5b}}{I_{6c}} -
  34300{{I_3}}^2{{I_{4b}}}^2{I_{4c}}{I_{6c}} -
  8925\imag {{I_3}}^2{I_{4b}}{{I_{4c}}}^2{I_{6c}} +
  4375\imag {{I_3}}^3{I_{4b}}{I_{5c}}{I_{6c}} + 12250{{I_3}}^3{{I_{4b}}}^2{I_{7c}}\\
- 5425\imag {{I_3}}^2{I_{5a}}{{I_{4b}}}^2{I_{5d}} +
  5390\imag {I_3}{{I_{4b}}}^4{I_{5d}} - 65100{{I_3}}^2{{I_{4b}}}^2{I_{5b}}{I_{5d}} +
  22050{{I_3}}^2{I_{5a}}{I_{4b}}{I_{4c}}{I_{5d}} -
  18165{I_3}{{I_{4b}}}^3{I_{4c}}{I_{5d}} \\
- 22050\imag {{I_3}}^2{I_{4b}}{I_{5b}}{I_{4c}}{I_{5d}} +
  36610\imag {I_3}{{I_{4b}}}^2{{I_{4c}}}^2{I_{5d}} -
  21560{I_3}{I_{4b}}{{I_{4c}}}^3{I_{5d}} +
  29925\imag {{I_3}}^2{{I_{4b}}}^2{I_{5c}}{I_{5d}} -
  39200{{I_3}}^2{I_{4b}}{I_{4c}}{I_{5c}}{I_{5d}} \\
+ 8575{{I_3}}^2{{I_{4b}}}^2{{I_{5d}}}^2 +
  6{{I_{4a}}}^3( 365\imag {{I_{4b}}}^3 - 11213{{I_{4b}}}^2{I_{4c}} +
     3615\imag {I_{4b}}{{I_{4c}}}^2 - 10395{{I_{4c}}}^3 +
     5{I_3}( {I_{5a}}( 137\imag {I_{4b}} + 252{I_{4c}} )\\
+ 14{I_{4c}}( 93\imag {I_{5b}} - 109{I_{5c}} )  +
        {I_{4b}}( -146{I_{5b}} - 165\imag {I_{5c}} + 392{I_{5d}} )  )  )  +
  8750{{I_3}}^3{I_{5a}}{I_{4b}}{I_{6d}} + 36750{{I_3}}^2{{I_{4b}}}^3{I_{6d}} \\
- 1750\imag {{I_3}}^3{I_{4b}}{I_{5b}}{I_{6d}} +
  30275\imag {{I_3}}^2{{I_{4b}}}^2{I_{4c}}{I_{6d}} -
  14000{{I_3}}^2{I_{4b}}{{I_{4c}}}^2{I_{6d}} - 21000{{I_3}}^3{I_{4b}}{I_{5c}}{I_{6d}} + 5{I_2}{{I_3}}^3( 1296{{I_{4a}}}^4 \\
+ 34475\imag {{I_3}}^5{I_{4b}} - 1666{{I_{4b}}}^4 +
     6\imag {{I_{4a}}}^3( 2087{I_{4b}} + 2534\imag {I_{4c}} )  -
     37093\imag {{I_{4b}}}^3{I_{4c}} + 37470{{I_{4b}}}^2{{I_{4c}}}^2 -
     13152\imag {I_{4b}}{{I_{4c}}}^3 \\
+ 14316{{I_{4c}}}^4 -
     2{{I_{4a}}}^2( -5925{{I_{4b}}}^2 + 20423\imag {I_{4b}}{I_{4c}} -
        20418{{I_{4c}}}^2 + 30{I_3}( 36{I_{5a}} + 213\imag {I_{5b}} - 239{I_{5c}} )
        )  + 5{I_3}( {I_{5a}}   ( 175{{I_{4b}}}^2 \\
- 764\imag {I_{4b}}{I_{4c}} - 2102{{I_{4c}}}^2 )  +
        2{{I_{4c}}}^2( -1683\imag {I_{5b}} + 3704{I_{5c}} )  +
        2{I_{4b}}{I_{4c}}( 4787{I_{5b}} - 283\imag {I_{5c}} - 2548{I_{5d}} )  -
        7\imag {{I_{4b}}}^2( 939{I_{5b}} \\
- 445\imag {I_{5c}} - 295{I_{5d}} )
        )  + {I_{4a}}( 37093\imag {{I_{4b}}}^3 - 36920{{I_{4b}}}^2{I_{4c}} +
        41476\imag {I_{4b}}{{I_{4c}}}^2 - 41244{{I_{4c}}}^3 +
        10{I_3}( {I_{5a}}( -333\imag {I_{4b}} + 1267{I_{4c}} )\\
+ 7{I_{4c}}( 423\imag {I_{5b}} - 734{I_{5c}} )  + {I_{4b}}( -2122{I_{5b}} + 1978\imag {I_{5c}} + 1323{I_{5d}} ) ) )  + 25{{I_3}}^2( 36{{I_{5a}}}^2 - 203\imag {I_{6a}}{I_{4b}} +  6{{I_{5b}}}^2 + 322{I_{4b}}{I_{6b}} \\
+ {I_{5a}}( 426\imag {I_{5b}} - 478{I_{5c}} )  -
        1014\imag {I_{5b}}{I_{5c}} + 932{{I_{5c}}}^2 + 567\imag {I_{4b}}{I_{6c}} -
        448{I_{4b}}{I_{6d}} )  )  +
  {{I_{4a}}}^2( 32550\imag {{I_3}}^5{I_{4b}} +
     6( 7308{{I_{4b}}}^4 \\
- 5975\imag {{I_{4b}}}^3{I_{4c}} +
        31787{{I_{4b}}}^2{{I_{4c}}}^2 - 5835\imag {I_{4b}}{{I_{4c}}}^3 + 9915{{I_{4c}}}^4
        )  - 5{I_3}( {I_{5a}}
         ( 1985{{I_{4b}}}^2 + 6836\imag {I_{4b}}{I_{4c}} + 5363{{I_{4c}}}^2 ) \\
 + {{I_{4c}}}^2( 9243\imag {I_{5b}} - 19384{I_{5c}} )  -
        2{I_{4b}}{I_{4c}}( 10944{I_{5b}} + 2375\imag {I_{5c}} - 5733{I_{5d}} )  +
        \imag {{I_{4b}}}^2( 7108{I_{5b}} + 5590\imag {I_{5c}} - 1547{I_{5d}} )) \\
+ 150{{I_3}}^2( 3{{I_{5a}}}^2 - 21\imag {I_{6a}}{I_{4b}} -
        94{{I_{5b}}}^2 + 84{I_{4b}}{I_{6b}} +
        {I_{5a}}( 46\imag {I_{5b}} - 48{I_{5c}} )  - 242\imag {I_{5b}}{I_{5c}} +
        143{{I_{5c}}}^2 + 133\imag {I_{4b}}{I_{6c}} - 70{I_{4b}}{I_{6d}} )  ) \\
- 18375\imag {{I_3}}^2{{I_{4b}}}^3{I_{6e}} + 14700{{I_3}}^2{{I_{4b}}}^2{I_{4c}}{I_{6e}} -
  {I_{4a}}( 175{{I_3}}^5{I_{4b}}( 22{I_{4b}} + 427\imag {I_{4c}} )  -
     6\imag ( 980{{I_{4b}}}^5 + 35756\imag {{I_{4b}}}^4{I_{4c}} \\
+ 16355{{I_{4b}}}^3{{I_{4c}}}^2 + 27203\imag {{I_{4b}}}^2{{I_{4c}}}^3 +
        3215{I_{4b}}{{I_{4c}}}^4 + 3941\imag {{I_{4c}}}^5 )  +
     5{I_3}( {I_{5a}}( -7693\imag {{I_{4b}}}^3 + 295{{I_{4b}}}^2{I_{4c}} -
           2706\imag {I_{4b}}{{I_{4c}}}^2 \\
- 6622{{I_{4c}}}^3 )  +
        14{{I_{4c}}}^3( 117\imag {I_{5b}} + 1054{I_{5c}} )  +
        2{I_{4b}}{{I_{4c}}}^2( 22249{I_{5b}} - 3435\imag {I_{5c}} - 6713{I_{5d}} )
            + 7{{I_{4b}}}^3( 3067{I_{5b}} + 1925\imag {I_{5c}} \\
- 2309{I_{5d}} )  + {{I_{4b}}}^2{I_{4c}}( -19041\imag {I_{5b}} - 4775{I_{5c}} +
           3969\imag {I_{5d}} )  )  +
     25{{I_3}}^2( -3868\imag {I_{4b}}{{I_{5b}}}^2 + 2597\imag {{I_{4b}}}^2{I_{6b}} +
        7{I_{6a}}{I_{4b}}( 33{I_{4b}} \\
- 91\imag {I_{4c}} )  +
        1722{{I_{5b}}}^2{I_{4c}} + 98{I_{4b}}{I_{6b}}{I_{4c}} +
        2{{I_{5a}}}^2( 71\imag {I_{4b}} + 63{I_{4c}} )  +
        3268{I_{4b}}{I_{5b}}{I_{5c}} - 714\imag {I_{5b}}{I_{4c}}{I_{5c}} +
        366\imag {I_{4b}}{{I_{5c}}}^2 \\
+ 2086{I_{4c}}{{I_{5c}}}^2 -
        2352{{I_{4b}}}^2{I_{6c}} + 441\imag {I_{4b}}{I_{4c}}{I_{6c}} -
        2{I_{5a}}( 7{I_{4c}}( -93\imag {I_{5b}} + 109{I_{5c}} )  +
           {I_{4b}}( 367{I_{5b}} + 499\imag {I_{5c}} - 196{I_{5d}} )  ) \\
 +  98\imag {I_{4b}}{I_{5b}}{I_{5d}} - 1078{I_{4b}}{I_{5c}}{I_{5d}} -
        329\imag {{I_{4b}}}^2{I_{6d}} - 980{I_{4b}}{I_{4c}}{I_{6d}} +
        343{{I_{4b}}}^2{I_{6e}} )  )  + 6125{{I_3}}^3{{I_{4b}}}^2{I_{7e}} .
 \end{multline*}
}

\pagebreak

\small
%%%%%%%%%%%%%%%%%%%%%%%%%%%%%%%%%%%%%%%%%%%%%%%%%%%%%%%%%%%%%%%%%%%%%%%%%%%%

\vspace{-10pt} \hspace{-20pt} {\hbox to 12cm{ \hrulefill }}
\vspace{-1pt}

{\footnotesize \hspace{-10pt} Institute of Mathematics and
Statistics, University of Troms\o, Troms\o\ 90-37, Norway.

\hspace{-10pt} E-mails: \quad kruglikov\verb"@"math.uit.no.} \vspace{-1pt}

\end{document}